\documentclass[10pt]{article}
\usepackage{amsfonts}

\usepackage{amsmath}
\usepackage{graphicx}
\usepackage[active]{srcltx}
\include{srctex}
\newtheorem{theorem}{Theorem}[section]

\newtheorem{corollary}[theorem]{Corollary}
 
\newtheorem{definition}[theorem]{Definition}

\newtheorem{lemma}[theorem]{Lemma}

\newtheorem{problem}[theorem]{Problem}
\newtheorem{proposition}[theorem]{Proposition}
\newtheorem{remark}[theorem]{Remark}

\newenvironment{proof}[1][Proof]{\textbf{#1.} }{\ \rule{0.5em}{0.5em}}

\def\R{{\Bbb R}}

\def\P{{\Bbb P}}

\def\al{{\alpha}}

\def\frac#1#2{{#1\over #2}}

\def\square{{\vcenter{\vbox{\hrule height.3pt
         \hbox{\vrule width.3pt height5pt \kern5pt
            \vrule width.3pt}
         \hrule height.3pt}}}}

\def\tlint{{- \kern-0.85em \int \kern-0.2em}}  % for textstyle
\def\dlint{{- \kern-1.05em \int \kern-0.4em}}  % for displays

\def\RR{\mathbb{R}}
\def\EE{\mathbb{E}}

\def\al{{\alpha}}

\def \eref#1{\hbox{(\ref{#1})}}

\def\EE{\mathbb{ E}\ }
\def \eref#1{\hbox{(\ref{#1})}}

\def\al{{\alpha}}

\begin{document}

\title{Feynman-Kac formula for fractional heat equations driven by fractional white
noises }
\author{Xia {\sc Chen}
\thanks{X. Chen is partially supported by NSF grant \#DMS-0704024}, \ \ Yaozhong {\sc Hu}\thanks{Y.  Hu is
partially supported by a grant from the Simons Foundation
\#209206.}\ \  and \ \ Jian {\sc Song} }\date{} \maketitle

\begin{abstract}
We obtain a Feynman-Kac formula for  the solution of a stochastic  fractional
 heat equation driven by fractional  noise.  One of the   main difficulties
  is to show the exponential integrability of some singular
 nonlinear  functional of symmetric stable L\'evy motion.  This difficulty will
 be  overcome by the technique of   logarithmic generating function of sub-additive processes,
 which explores the scaling property of the stable L\'evy motion.
 This Feynman-Kac formula   is applied to obtain the H\"older continuity and the moment formula
of the solution.
\end{abstract}

\begin{quote} {\footnotesize
}

\end{quote}

\setcounter{equation}{0}

\section{Introduction}
Let $0<\al\le 2$ and let $-(-\Delta)^{\frac\al2}$ be the fractional  Laplacian.
In this paper, we shall obtain a Feynman-Kac representation for a solution of
the following stochastic partial differential equation driven by   fractional noise
\begin{equation}\label{e.1.1}
\begin{cases}
\dfrac{\partial u}{\partial t}=-(-\Delta)^{\frac{\alpha}{2}}u+u\dfrac{\partial ^{d+1}W}{%
\partial t\partial x_1 \cdots \partial x_d} \\
u(0,x)=f(x)\,,%
\end{cases}
\end{equation}%
where $W(t,x)$ is a fractional Brownian sheet with Hurst
parameters $H_{0}$ in time and $(H_1, \dots, H_d)$ in space,
respectively.  The product between $u$ and the noise is
understood in Stratonovich sense (see definition \ref{def2} for detail).
More specifically, let   $H=(H_0, H_1, \cdots, H_d)$ satisfy
\begin{equation}
\frac12<H_0,H_1,\cdots,H_d<1,\label{Hurst1}
\end{equation}%
and
\begin{equation}
2H_0+\frac1\alpha\sum_{i=1}^d(2H_i-2)>1\,.\label{Hurst2}
\end{equation}%
Under the above conditions \eref{Hurst1} and \eref{Hurst2}, we shall show that the
following expression
\begin{equation}
u(t,x)=E^{X}\left[ f(X_{t}^{x})\exp \left(
\int_{0}^{t}\int_{\mathbb{R}^d} \delta
(X_{t-r}^{x}-y)W(dr,dy)\right) \right]    \label{feynman}
\end{equation}%
is well-defined and is a weak solution   (see Definition \ref{def}) to the above  equation
\eref{e.1.1},
where $E^{X}$ denotes the expectation with respect to the
$d$-dimensional symmetric $\alpha$-stable L\'evy motion $%
X_{t}^{x}$, and $\delta$ denotes the Dirac delta function.
The identity \eref{feynman} is called   the Feynman-Kac
formula for  \eref{e.1.1}.

To justify the Feynman-Kac formula, there are three tasks to complete. The first one
is to show that  $V^{\al, H}_{t,x}:=\int_{0}^{t}\int_{\mathbb{R}^d} \delta
(X_{t-r}^{x}-y)W(dr,dy)$ is a well-defined random variable since the integrand involves a Dirac
delta function
(in the future when there is no ambiguity we shall omit the explicit dependence on
 $\al, H$ in
$V^{\al, H}_{t,x}$). The second task is to prove the exponential integrability
(which is  one of the main difficulties)  of
$V^{\al, H}_{t,x} $ so that $u(t, x)$ given by \eref{feynman} is well-defined.  The third
task is to prove that such defined $u(t, x)$ is indeed a (weak) solution to \eref{e.1.1}.

%When $\displaystyle \dfrac{\partial ^{d+1}W}{%
%\partial t\partial x_1 \cdots \partial x_d}$ in \eref{e.1.1} is replaced by a (deterministic)
%continuous function $c(t,x)$, this
%is the classical Feynman-Kac formula, which has been widely studied
%(see \cite{fred}).
%
%To prove our main results, one of the main difficulties is the exponential integrability
%of $\int_{0}^{t}\int_{\mathbb{R}^d} \delta
%(X_{t-r}^{x}-y)W(dr,dy)$, appeared in \eref{feynman}.

When  $-(-\Delta)^{\al/2}$ in \eref{e.1.1}  is replaced by $\frac12\Delta$,
the corresponding Feynman-Kac formula  was first studied   in \cite{hns},
where the   $\alpha$-stable L\'evy motion $X_t^x$ in \eref{feynman} is replaced
by the standard Brownian motion. In this case (and under our conditions
\eref{Hurst1} and \eref{Hurst2}), the above first task is relatively easier.
One of the main difficulties is   the second task which is to show the exponential
integrability of $V_{t, x}^{2, H}$ (The number $2$ means that the
$X_t^x$ is replaced by  Brownian motion $B_t+x$).
The idea in \cite{hns} is to obtain precise moments of
$V_{t, x}^{2, H}$, which appealed  to an idea of Le Gall \cite{legall}.
To complete the third task, the Malliavin calculus is used.

Still for the Laplacian, the fractional noise with   Hurst
parameter
 $H_0<\frac12$ has been studied in a recent paper \cite{hln}.  Namely, it is  studied
 the stochastic heat equation
driven by $\dfrac{\partial W}{\partial t}(t,x)$, where $\{W(t,x),,t\ge0,x\in\R^d\}$
is Gaussian random field with covariance function \[E(W(t,x)W(s,y))
=\frac12(t^{2H_0}+s^{2H_0}-|t-s|^{2H_0})Q(x,y),\] where
$H_0\in(\frac14,\frac12)$ and $Q(x,y)$ is a locally H\"older
continuous function. Heuristically, comparing with the noise
in \cite{hns}, the noise $\dfrac{\partial W}{\partial t}(t,x)$ is
rougher in time, while as a compensation, is more regular in space.
Technically, with less regularity in time, we need more regularity
in space to make sense of the term $\displaystyle\int_0^t\int_{\R^d}
\delta(B_{t-r}^x-y) W(dr,y)=\int_0^t\int_{\R^d}W(dr,B_{t-r}^x)$,
where $B$ is a $d$-dimension Brownian
motion. The main task of
that paper is to complete
the first task.

Return to our fractional Laplacian case. We
shall only deal with the case $H_0>1/2$.
In this case it is expected that the main difficulty
is also to complete the second task.  Presumably, one can still try to use
the moment method.%
%  the techniques we shall use  to prove
%the exponential integrability of $\int_{0}^{t}\int_{\mathbb{R}^d} \delta
%(X_{t-r}^{x}-y)W(dr,dy)$ (which is one
%of the main ingredients to justify the Feynman-Kac formula \eref{feynman})
%is completely different  and is much simpler.  In Theorem 3.3 of the paper \cite{hns},
% the authors used moment method to explore the exponential integrability,
%  the idea of which was inspired from \cite{legall}, and we don't know
%  whether this approach can be applied to the case of stable processes.
However, we are not sure if it   works  or not since
it is very technical and long even in the Brownian motion case.
We used  freely the high moments of the Brownian motions in \cite{hns}.
However,  the $\al$-stable L\'evy motion has no second or
higher moments.
Instead of following the approach of \cite{hns},
we shall use the logarithmic moment generating function
  of sub-additive processes (see Theorem 1.3.5 in
  \cite{chen}). This technique is to explore
  the scaling property of the $\al$-stable L\'evy motion to obtain an concentration
  inequality and then
  the concentration
  inequality is used to obtain  the desired exponential integrability.
  This approach will largely simplify our computation
even in the standard Brownian motion case. In general,
we shall follow the approach of \cite{hns} and \cite{hln}
to fulfill the third task, i.e., to approximate the noise in \eref{e.1.1} by 
use Malliavin calculus to show that \eref{feynman} is indeed a weak solution to \eref{e.1.1}. As an application, we shall show the H\"older continuity of the solution given by the Feynman-Kac formula (\ref{feynman}) using the standard Kolmogorov continuity criterion as in \cite{hns}.

If we formally set $\al=2$,  then we obtain  exactly
same condition for Feynman-Kac formula to be valid as in
 \cite{hns}.
Notice that $2H_i-2<0$. The conditions for $H$ is more
restrictive when $\al$ becomes smaller. This is consistent to the intuition.
because  when $\alpha$ gets smaller, the transition probability
density of fractional Laplacian
$-(-\Delta)^{\alpha/2}$
gets less regular, and hence the range of $H_i$'s should get more
restrictive.
%(meaning that the Hurst parameters are bigger) to obtain a Feynman-Kac formula.
We can also prove that under the condition \eref{Hurst1},
the condition \eref{Hurst2} is optimal (see Proposition
\ref{best}). Thus we see that under the condition
\eref{Hurst1},   no matter how  one choose  the parameters $H_i$'s
(as close to $1$ as possible),
there is a value of $\alpha$ below which the Feyman-Kac formula can not
obtained.  On the other hand,
the condition \eref{Hurst2} only allows us  to
consider the case
$H_0>\frac12.$

There are several problems related to this topic which are
also our future work. For the reader's interest, we list some of them here.
(1) Uniqueness: We still don't know whether the solution to (\ref{e.1.1})
given by the Feynman-Kac formula (\ref{feynman}) is the unique solution
(in some function space). (2) Extension of the generator: We may extend
the fractional Laplacian $-(-\Delta)^{\frac\alpha2}$ to   more general
operator  $L$ which generates a Markov process. For fractional
Laplacian,  the condition for Feynman-Kac formula involves $\al$
in a very nice way expressed by
\eref{Hurst2}. It is very interesting to know what will be
the expression of such condition in  more general case.
  (3) General noise: (a) It is also interesting
 to  obtain a parallel result in \cite{hln}, i.e., consider
 the Gaussian random field  which has temporal Hurst parameter $H_0<1/2$;
  (b) we may also consider the case   $H_i\in(0,\frac12), i=1,\cdots, d$
 and $2H_0+\frac1\alpha\sum_{i=1}^d(2H_i-2)>1$;  (c) Extend the
 fractional Gaussian white noise to some other general  Gaussian field.
 (4) Existence of the solution: We may explore the conditions under which equation
 (\ref{e.1.1}) has a (unique) mild solution (see Definition \ref{def1}),
 which may not have a Feynman-Kac type of representation, when the product
 between $u$ and the noise is in It\^o-Skorohod
 sense.

The paper is organized as follows.  In Section 2, we very briefly
present some preliminary material on Malliavin calculus and stable L\'evy
motion that we need  to fix some notations.  In Section 3, we
give a definition to the stochastic integral in \eref{feynman}.  This is
necessary since the integrand involves the Dirac delta function.
We also show the exponential integrability of the mentioned stochastic integral.
In Section 4, we use Malliavin calculus to
prove that \eref{feynman} is a solution to \eref{e.1.1}
in a weak sense.  Section 5 is an application of Feynman-Kac formula \eref{feynman}, where we prove that the solution $u(t,x)$ given by \eref{e.1.1} has a H\"older continuous version. The main part of the paper assumes that the stochastic
integral involved in \eref{e.1.1} is in the sense of Stratonovich.  In Section 6, we obtain the Feynman-Kac formula as the unique solution to (1.1) where the stochastic integral in \eref{e.1.1} is in the sense of It\^o-Skorohod. We also obtain a Feynman-Kac formula to represent the It\^o-Wiener chaos coefficients of the solution.  A formula for the moments of the solution is also given. Finally in Section 7, we present some lemmas used in this paper.

\setcounter{equation}{0}

\section{Preliminaries}
\subsection{Fractional Brownian sheet and Malliavin calculus}
Let $W=\{W(t,x),t\geq 0,x\in \mathbb{R}%
^d\}$ be a fractional Brownian sheet with Hurst parameters $%
H_{0}\in (1/2,1)$ in time variable and $H_{i}\in (1/2, 1), i=1,\dots,d,$ in space variables, i.e., $W$ is a Gaussian random field with mean zero and  covariance function defined by 
\begin{equation*}
E(W(t,x)W(s,y))=R_{H_{0}}(s,t)\prod_{i=1}^{d }R_{H_{i}}(x_i, y_i),
\end{equation*}%
where for \ any $H\in (0,1),$ 
\begin{equation*}
R_{H}(s,t)=\dfrac{1}{2}(|t|^{2H}+|s|^{2H}-|t-s|^{2H}),
\end{equation*}
which is the covariance
function of the fractional Brownian motion with Hurst parameter $H$.

In this subsection, we shall briefly recall the definition of  Wiener-It\^o integral with respect to the fractional Brownian sheet and the related Malliavin calculus we shall use in this article. 
 
Let $\mathcal{E}$  be the linear span of the indicator functions
 $I_{(s,t]\times (x,y]}$ in $\mathbb{R}_{+}\times \mathbb{R%
}^{d}$, where $(x, y]=(x_1, y_1]\times \cdots \times (x_d, y_d]$.
Define  the inner product in  $\mathcal{E}$ 
\begin{align} \label{r1}
\langle \phi,\psi\rangle _{\mathcal{H}}=\alpha _{H}   \times \int_{\mathbb{R}_{+}^{2}\times \mathbb{R}^{2d}}\phi
(s,x)\psi    
(t,y)|s-t|^{2H_{0}-2}\prod_{i=1}^{d}|x_{i}-y_{i}|^{2H_{i}-2}dsdtdxdy,
\end{align}
where \ $\alpha _{H}=\prod_{i=0}^{d}H_{i}(2H_{i}-1)$, and the Hilbert space $\mathcal H$ is the closure of $\mathcal E$ with respect to this inner product.
Note that 
\[\langle I_{(0,s]\times (0,x]}, I_{(0,t]\times (0,y]}\rangle_{\mathcal{H}}=R_{H_0}(s,t)\times \prod_{i=1}^d R_{H_i}(x_i, y_i)=E(W(s,x)W(t,y)).\]
Furthermore, $\mathcal{H}$ contains the class of \ measurable functions $\phi $ on \ $%
\mathbb{R}_{+}\times \mathbb{R}^{d}$ such that%
\begin{equation}
\int_{\mathbb{R}_{+}^{2}\times \mathbb{R}^{2d}}|\phi (s,x)\phi
(t,y)||s-t|^{2H_{0}-2}\prod_{i=1}^{d}|x_{i}-y_{i}|^{2H_{i}-2}dsdtdxdy<\infty.
\label{r2}
\end{equation}

 The mapping $W:I_{(0,t]\times (0,x]}\rightarrow W(t,x)$ extends to a linear
isometry between $\mathcal{H}$ and the Gaussian space spanned by
$W$. This isometry defines Wiener-It\^o integral denoted by
\begin{equation*}
W(\phi )=\int_{0}^{\infty }\int_{\mathbb{R}^{d}}\phi
(t,x)W(dt,dx), \text{ for } \phi \in \mathcal{H}.
\end{equation*}

We will denote by $D$ the Malliavin derivative. That is, if $F$ is a smooth and cylindrical
random variable of the
form%
\begin{equation*}
F=f(W(\phi _{1}),\ldots ,W(\phi _{n})),
\end{equation*}%
$\phi _{i}\in \mathcal{H}$, $f\in C_{p}^{\infty }(\mathbb{R}^{n})$
($f$ and
all its partial derivatives have polynomial growth), then $DF$ is the $%
\mathcal{H}$-valued random variable defined by
\begin{equation*}
DF=\sum_{j=1}^{n}\frac{\partial f}{\partial x_{j}}(W(\phi
_{1}),\ldots ,W(\phi _{n}))\phi _{j}.
\end{equation*}%
The operator $D$ is closable from $L^{2}(\Omega )$ into $L^{2}(\Omega ;%
\mathcal{H})$ and we define the Sobolev space $\mathbb{D}^{1,2}$
as the closure of the space of smooth and cylindrical random
variables under the
norm%
\begin{equation*}
\left\| DF\right\| _{1,2}=\sqrt{E(F^{2})+E(\left\| DF\right\| _{\mathcal{H}%
}^{2})}\, .
\end{equation*}%
Let  $\delta $ be the adjoint of the derivative operator $D$,
determined  by duality formula
\begin{equation}
E(\delta (u)F)=E\left( \left\langle DF,u\right\rangle
_{\mathcal{H}}\right) , \label{dua}
\end{equation}%
for any $F\in \mathbb{D}^{1,2}$ and any element $u\in $ $L^{2}(\Omega ;%
\mathcal{H})$ in the domain of $\delta $. The operator $\delta $
is also called the Skorohod integral.  If $DF$ and $u$  are almost
surely  measurable functions on $\mathbb{R}_+ \times \mathbb{R}^d$
satisfying condition  (\ref{r2}), then the duality formula
(\ref{dua}) can be written as:
\begin{eqnarray*}
&& E\left( \delta(u)F\right) = \alpha_H  \\
&&\times  E\left( \int _ {\mathbb{R}^2_+\times \mathbb{R}^{2d}}
D_{s,x} F u(t,y)
|s-t|^{2H_{0}-2}\prod_{i=1}^{d}|x_{i}-y_{i}|^{2H_{i}-2}dsdtdxdy
\right).
\end{eqnarray*}
The following formula 
formula plays an essential role in the proof of Theorem \ref{th1}:
\begin{equation}
FW(\phi )=\delta (F\phi )+\left\langle DF,\phi \right\rangle
_{\mathcal{H}}, \label{a5}
\end{equation}%
for any $\phi \in \mathcal{H}$ and any random variable $F$ in the
Sobolev space $\mathbb{D}^{1,2}$.

We refer to Nualart
\cite{nualart} for a detailed account on the Malliavin calculus
with respect to a Gaussian process.

\subsection{Symmetric $\alpha$-stable L\'evy motion}
In this section we recall symmetric stable distribution and symmetric $\alpha$-stable L\'evy motion. For more general and detailed results about stable processes, we refer to \cite{ST}.\\
A random variable $X$ is said to be symmetric $\alpha$-stable if there are parameters $0<\alpha\le2$ and $\sigma\ge0$
such that its characteristic function satisfies  \[Ee^{i\theta X}=e^{-\sigma^\alpha|\theta|^\alpha},\] and we will denote that $X\sim S(\alpha,\sigma).$
Notice that when $\alpha=2, X$ is a Gaussian random variable. When $\alpha\in(0,2),$ we have $E|X|^p<\infty$ if $-1<p<\alpha$, and $E|X|^p=\infty$ if $p\ge \alpha.$\\
A stochastic process $\{X(t), t\ge0\}$ is called a symmetric $\alpha$-stable L\'evy motion if
\begin{itemize}
\item[(1)] $X(0)=0 $ a.s..
\item[(2)] $X$ has independent increments.
\item[(3)] $X(t)-X(s)\sim S(\alpha, (t-s)^\frac 1\alpha)$ for any $0\le s<t<\infty$ and for some $0<\alpha\le 2.$
\end{itemize}

As for the classical Brownian motion case, for the following fractional linear heat equation,
\begin{equation*}
\begin{cases}
\dfrac{\partial u}{\partial t}= -(-\Delta)^{\frac \alpha2}u+uc(t,x) \\
u(0,x)=f(x).
\end{cases}
\end{equation*}
where $f$ is a bounded measurable function and $c(t,x)$ is a continuous function of $(t,x)\in [0,\infty)\times\mathbb R^d$,
we also have the Feynman-Kac formula
\begin{equation*}
u(t,x) =E^{X}\left( f(X_{t}^{x})\exp \left(
\int_{0}^{t}c(t-s,X_{s}^{x})ds\right)
\right),
\end{equation*}%
where $X^x$ is a $d$-dimensional symmetric $\alpha$-stable L\'evy motion starting at $x$.

Throughout the paper $C$ will denote a positive constant which may vary
from one formula to another one.

\setcounter{equation}{0}

\section{Definition and exponential integrability of  the generalized stochastic convolution}

For any $\varepsilon >0$ we denote \ by $p_{\varepsilon }(x)$ the $d$%
-dimensional heat kernel
\begin{equation*}
p_{\epsilon }(x)=(2\pi \varepsilon )^{-\frac{d}{2}}e^{-\frac{|x|^{2}}{%
2\epsilon }},\;x\in \mathbb{R}^{d}.
\end{equation*}%
On the other hand, for any $\delta >0$ we define the function
\begin{equation*}
\varphi _{\delta }(t)=\dfrac{1}{\delta }I_{[0,\delta ]}(t).
\end{equation*}%
Then, $\varphi _{\delta }(t)p_{\varepsilon }(x)$ provides an
approximation of the Dirac delta function $\delta (t,x)$ as
$\varepsilon $ and $\delta $ tend to zero. We denote by
$W^{\epsilon ,\delta }$ the approximation of the fractional
Brownian sheet $W(t,x)$ defined by
\begin{equation}
W^{\epsilon ,\delta
}(t,x)=\int_{0}^{t}\int_{\mathbb{R}^{d}}\varphi _{\delta
}(t-s)p_{\epsilon }(x-y)W(s,y)dsdy\,.  \label{e1}
\end{equation}

Fix $x\in \mathbb{R}^{d}$ and $t>0$. Suppose that $X=\{X_{t},t\geq
0\}$ is a
$d$-dimensional symmetric $\alpha$-stable L\'evy motion independent of $W$. We denote by $%
X_{t}^{x}=X_{t}+x$ the symmetric $\alpha$-stable L\'evy motion starting at the
point $x$. \ We are going to define the random variable
$\int_{0}^{t}\int_{\mathbb{R}^{d}}\delta (X_{t-r}^{x}-y)W(dr,dy)$
by approximating the Dirac delta function $\delta
(X_{t-r}^{x}-y)$ by%
\begin{equation}
A_{t,x}^{\epsilon ,\delta }(r,y)=\int_{0}^{t}\varphi _{\delta
}(t-s-r)p_{\varepsilon }(X_{s}^{x}-y)ds.  \label{e2}
\end{equation}%
We will show that for any $\varepsilon >0$ and $\delta >0$\ $\
$the function \ $A_{t,x}^{\epsilon ,\delta }$ belongs to the space
$\mathcal{H}$ almost
surely, and the family of random variables%
\begin{equation}
V_{t,x}^{\varepsilon ,\delta }=\ \int_{0}^{t}\int_{\mathbb{R}%
^{d}}A_{t,x}^{\epsilon ,\delta }(r,y)W(dr,dy)\,.  \label{e3}
\end{equation}%
converges in $L^{2}$ as $\varepsilon $ and $\delta $ tend to zero.

The specific approximation chosen here will allow us in Section 4 to
construct an approximate Feynman-Kac formula with the random
potential   $\dot{W}^{\epsilon ,\delta }(t,x)$ given in (\ref{pot}).
Moreover, this approximation has the useful properties  proved in
 Lemmas \ref{lemma2} and \ref{lemma3}. We may  use  other types
of approximation schemes with similar results.   Also, we can
restrict ourselves to the special case ${\delta }={\varepsilon }$,
but the slightly more general case considered here does not need
any additional effort.

Along the paper we denote by $E^X(\Phi(X,W))$ (resp. by
$E^W(\Phi(X,W))$) the expectation of a functional $\Phi(X,W)$ with
respect to $X$ (resp. with respect to $W$). We will use $E$ for
the composition $E^XE^W$, and also in case of a random variable
depending only on $X$ or $W$.

The $L^2$ space in the paper means $L^2(\Omega) $, where $\Omega=\Omega_W\times\Omega_X$ is the whole sample space generated by both the fractional Brownian sheet $W$ and the stable process $X$, if there is no other explanatory note in the context.

\begin{theorem}
\label{teo1}Suppose that \ $2H_{0}+\frac{1}{\alpha}\sum_{i=1}^d (2H_{i}-2)>1$. Then, for any $%
\varepsilon>0 $ and $\delta >0$, $A_{t,x}^{\varepsilon ,\delta }$
defined in
(\ref{e2}) belongs to $\mathcal{H}$ and the family of random variables $%
V_{t,x}^{\epsilon ,\delta }$ defined in (\ref{e3}) converges in
$L^{2}$ to a limit denoted by
\begin{equation}
V_{t,x}=\int_{0}^{t}\int_{\mathbb{R}^d}\delta
(X_{t-r}^{x}-y)W(dr,dy)\,. \label{e.3.4}
\end{equation}
Conditional on $X$, $V_{t,x}$ is a Gaussian random variable with
mean $0$
and variance%
\begin{equation}
\mathrm{Var}(V_{t,x}|X) =\alpha _H \int_{0}^{t}\int_{0}^{t}|r
-s|^{2H_{0}-2} \prod_{i=1}^d \left| X^i_{r }-X^i_{s}\right|
^{2H_{i}-2}dr ds\,. \label{e.2.14}
\end{equation}
\end{theorem}

\begin{proof}
Fix $\varepsilon $, $\varepsilon ^{\prime }$, $\delta $ and
$\delta ^{\prime
}>0$. Let us compute the\ inner product%
\begin{eqnarray}
\left\langle A_{t,x}^{\varepsilon ,\delta },A_{t,x}^{\varepsilon
^{\prime },\delta ^{\prime }}\right\rangle _{\mathcal{H}}
&=&\alpha _{H}\int_{[0,t]^{4}}\int_{\mathbb{R}^{2d}}p_{\epsilon
}(X_{s}^{x}-y)p_{\epsilon ^{\prime }}(X_{r}^{x}-z)  \notag \\
&&\times \varphi _{\delta }(t-s-u)\varphi _{\delta ^{\prime
}}(t-r-v)  \notag
\\
&&\!\!\!\!\!\!\!\!\!\times
|u-v|^{2H_{0}-2}\prod_{i=1}^{d}|y_{i}-z_{i}|^{2H_{i}-2}dydzdudvdsdr.
\label{e4}
\end{eqnarray}%
By Lemmas \ref{lemma2} and \ref{lemma3} we have the estimate
\begin{eqnarray}
&&\int_{[0,t]^{2}}\int_{\mathbb{R}^{2d}}p_{\epsilon
}(X_{s}^{x}-y)p_{\epsilon ^{\prime }}(X_{r}^{x}-z)  \notag \\
&&\times \varphi _{\delta }(t-s-u)\varphi _{\delta ^{\prime
}}(t-r-v)|u-v|^{2H_{0}-2}\prod_{i=1}^{d}|y_{i}-z_{i}|^{2H_{i}-2}dydzdudv
\notag \\
&\leq &C|s-r|^{2H_{0}-2}\prod_{i=1}^{d}\left|
X_{s}^{i}-X_{r}^{i}\right| ^{2H_{i}-2},  \label{e5}
\end{eqnarray}%
where and in what follows  $C>0$ denotes a constant dependent only on
$H=(H_0, H_1, \cdots, H_d)$ (independent of
$\epsilon$, $\delta$, $r$ and $s$). The expectation of this random variable is
integrable in $[0,t]^{2}$ because%
\begin{eqnarray}
&&E^{X}\int_{0}^{t}\int_{0}^{t}|s-r|^{2H_{0}-2}\prod_{i=1}^{d}\left|
X_{s}^{i}-X_{r}^{i}\right| ^{2H_{i}-2}dsdr  \notag \\
&=&\prod_{i=1}^{d}E|X_1
|^{2H_{i}-2}\int_{0}^{t}\int_{0}^{t}|s-r|^{2H_{0}+%
\sum_{i=1}^{d}\frac{1}{\alpha}(2H_{i}-2)-2}dsdr  \notag \\
&=&\frac{2\prod_{i=1}^{d}E|\xi |^{2H_{i}-2}t^{\kappa +1}}{\kappa
\left( \kappa +1\right) }\ <\infty ,  \label{3.8}
\end{eqnarray}%
where
\begin{equation}
\kappa =2H_{0}+\frac{1}{\alpha}\sum_{i=1}^{d}(2H_{i}-2)-1>0.
\label{kappa}
\end{equation}%
and $\xi $ is a standard symmetric $\alpha$-stable random variable.

As a consequence, taking the mathematical expectation with respect
to $X$ in
Equation (\ref{e4}), letting $\varepsilon =\varepsilon ^{\prime }$ and $%
\delta =\delta ^{\prime }$ and using the estimates (\ref{e5}) and (\ref{3.8}%
) yields%
\begin{equation*}
E^{X}\left\| A_{t,x}^{\varepsilon ,\delta }\right\|
_{\mathcal{H}}^{2}\leq C.
\end{equation*}%
This implies that $A_{t,x}^{\varepsilon,\delta}\in L^2(\Omega_X;\mathcal H)$, and hence for almost all $\omega\in \Omega_X$, $A_{t,x}^{\varepsilon ,\delta }$
belongs to the space $\mathcal{H}$, for all $\varepsilon $ and
$\delta >0$. Therefore, the random variables $V_{t,x}^{\varepsilon
,\delta }=W(A_{t,x}^{\varepsilon ,\delta })$ are well defined and
we have
\begin{equation*}
E^{X}E^{W}(V_{t,x}^{\epsilon ,\delta }V_{t,x}^{\epsilon ^{\prime
},\delta ^{\prime }})=E^{X}\left\langle A_{t,x}^{\varepsilon
,\delta
},A_{t,x}^{\varepsilon ^{\prime },\delta ^{\prime }}\right\rangle _{\mathcal{%
H}}.
\end{equation*}%
For any $s\neq r$ and $X_{s}\neq X_{r}$, as $\varepsilon $,
$\varepsilon ^{\prime }$, $\delta $ and $\delta ^{\prime }$ tend
to zero, the left-hand
side of the inequality \ (\ref{e5}) converges to $|s-r|^{2H_{0}-2}%
\prod_{i=1}^{d}\left| X_{s}^{i}-X_{r}^{i}\right| ^{2H_{i}-2}$.
Therefore, by dominated convergence theorem we obtain that \
$E^{X}E^{W}(V_{t,x}^{\epsilon
,\delta }V_{t,x}^{\epsilon ^{\prime },\delta ^{\prime }})$ converges to $%
\Sigma _{t}$ as $\varepsilon $, $\varepsilon ^{\prime }$, $\delta $ and $%
\delta ^{\prime }$ tend to zero, where
\begin{equation*}
\Sigma _{t}=\frac{2\alpha _{H}\prod_{i=1}^{d}E|\xi |^{2H_{i}-2}t^{\kappa +1}%
}{\kappa \left( \kappa +1\right) }.
\end{equation*}%
Thus we obtain
\begin{equation*}
E\ \left( V_{t,x}^{{\varepsilon },{\delta }}-V_{t,x}^{{\varepsilon
}^{\prime
},{\delta }^{\prime }\ }\right) ^{2}=E\ \left( V_{t,x}^{{\varepsilon },{%
\delta }\ }\right) ^{2}-2E\ \left( V_{t,x}^{{\varepsilon },{\delta }%
}V_{t,x}^{{\varepsilon }^{\prime },{\delta }^{\prime }}\right) +E\
\left( V_{t,x}^{{\varepsilon }^{\prime },{\delta }^{\prime }\
}\right) ^{2}\rightarrow 0\,.
\end{equation*}%
This implies that \ $V_{t,x}^{\epsilon _{n},\delta _{n}}$ is a
Cauchy sequence in $L^{2}$ for all  sequences $\varepsilon _{n}$
and $\delta _{n}$ converging to zero. As a consequence,
$V_{t,x}^{\epsilon _{n},\delta _{n}}$ converges in $L^{2}$ to a
limit denoted by $V_{t,x}$, which does not depend on the choice of
the sequences   $\varepsilon _{n}$ and $\delta _{n}$.
Furthermore, conditional on $X$, $V_{t,x}$ is Gaussian, since conditional on $X$, $V_{t,x}^{\varepsilon,\delta}$ is Gaussian. Finally, by a similar argument we can show  (\ref{e.2.14}).
\end{proof}

The next result provides the exponential integrability of the
random variable $V_{t,x}$ defined in $\ $(\ref{e.3.4}).
\begin{theorem}\label{exp}
\label{teo2} Suppose that $2H_0+\frac{1}{\alpha} \sum_{i=1}^d (2H_{i}-2)>1$. Then, for any $%
\lambda \in \mathbb{R}$, we have
\begin{equation}
E \exp \left( \lambda \int_{0}^{t}\int_{\mathbb{R}^d}\delta
(X_{t-r}^{x}-y)W(dr,dy) \right) <\infty \,.  \label{e21}
\end{equation}
\end{theorem}

\begin{proof}
The proof will be done in several steps.

\medskip \noindent \textbf{Step 1} \ By (\ref{e.2.14}),
the fact that $V_{t,x}$ is Gaussian conditional on $X$ and Fubini's  theorem, we obtain
\begin{eqnarray*}
Ee^{\lambda V_{t,x}}
&=&E ^X E\left[ e^{\lambda V_{t,x}} | X\right] =E ^X    e^{{\rm Var} \left(\lambda V_{t,x}
|X\right)} \\
&=& E^{X}\exp \left( \frac{\lambda ^{2}}{2}\alpha
_{H}\int_{0}^{t}\int_{0}^{t}|s-r|^{2H_{0}-2}%
\prod_{i=1}^{d}|X_{s}^{i}-X_{r}^{i}|^{2H_{i}-2}dsdr\right).
\end{eqnarray*}
Define
\begin{equation}
Y_t=\int_{0}^{t}\int_{0}^{t}|s-r|^{2H_{0}-2}%
\prod_{i=1}^{d}|X_{s}^{i}-X_{r}^{i}|^{2H_{i}-2}dsdr.  \label{e9}
\end{equation}

The scaling property of the stable L\'evy motion yields
\begin{equation}\label{scaling}
Y_t \overset{d}{=} t^{\kappa+1}Y_1,
\end{equation}
where
$\kappa $ has been defined in (\ref{kappa}).
Hence
\begin{equation}
Ee^{\lambda V_{t,x}}=Ee^{\mu Y_1},  \label{e9a}
\end{equation}%
where $\mu =\frac{\lambda ^{2}}{2}\alpha _{H}t^{\kappa +1}$. Then, it suffices to show that the random variable $Y_1$ has exponential moments of all orders.\\
\medskip \noindent
\textbf{Step 2}  Let $$Z_t=Y_t^{\frac12}=\left(\int_{0}^{t}\int_{0}^{t}|s-r|^{2H_{0}-2}%
\prod_{i=1}^{d}|X_{s}^{i}-X_{r}^{i}|^{2H_{i}-2}dsdr\right)^\frac{1}{2}.$$
We use the the identity
$$
\vert s-r\vert^{2H_0-2}=C_0\int_{\R}\vert s-u\vert^{2H_0-3\over 2}
\vert r-u\vert^{2H_0-3\over 2}du
$$
$$
\vert X_s^i-X_r^i\vert^{2H_i-2}=C_i\int_{\R}\vert X_s^i-x\vert^{2H_i-3\over 2}
\vert X_r^i-x\vert^{2H_i-3\over 2}dx\hskip.2in i=1,\cdots, d
$$
where $C_i$ only depends on $H_i$ for $i=0,1,\cdots,d.$\\
We have
$$
Z_t=\bigg(\int_{\R\times\R^d}\xi_t^2(u, x_1,\cdots, x_d)
du dx_1\cdots dx_d\bigg)^{1/2}
$$
where
$$
\xi_t(u, x_1,\cdots, x_d)=\Big(\prod_{i=0}^dC_i\Big)
\int_0^t\vert s-u\vert^{2H_0-3\over 2}\prod_{i=1}^d
\vert X_s^i-x_i\vert^{2H_i-3\over 2}ds
$$
For $t_1, t_2>0$, by triangular inequality
$$
Z_{t_1+t_2}\le Z_{t_1}+\bigg(\int_{\R\times\R^d}
\Big[\xi_{t_1+t_2}(u, x_1,\cdots, x_d)-\xi_{t_1}(u, x_1,\cdots, x_d)\Big]^2
du dx_1\cdots dx_d\bigg)^{1/2}
$$
Write $\widetilde{X}_s^i=X_{t_1+s}^i-X_{t_1}^i$.
\begin{align*}
&\xi_{t_1+t_2}(u, x_1,\cdots, x_d)-\xi_{t_1}(u, x_1,\cdots, x_d)\cr
=&\Big(\prod_{i=0}^dC_i\Big)\int_{t_1}^{t_1+t_2}
\vert s-u\vert^{2H_0-3\over 2}\prod_{i=1}^d
\vert X_s^i-x_i\vert^{2H_i-3\over 2}ds\cr
=&\Big(\prod_{i=0}^dC_i\Big)\int_{0}^{t_2}
\vert s+t_1-u\vert^{2H_0-3\over 2}\prod_{i=1}^d
\vert \widetilde{X}_{s}^i+X_{t_1}^i-x_i\vert^{2H_i-3\over 2}ds\cr
\end{align*}
By translation invariance,

\begin{align*}
&\int_{\R\times\R^d}
\Big[\xi_{t_1+t_2}(u, x_1,\cdots, x_d)-\xi_{t_1}(u, x_1,\cdots, x_d)\Big]^2\cr
=&\int_{\R\times\R^d}\widetilde{\xi}_{t_2}^2(u, x_1,\cdots, x_d)
du dx_1\cdots dx_d\cr
\end{align*}
where
$$
\widetilde{\xi}_{t_2}^2(u, x_1,\cdots, x_d)
=\Big(\prod_{i=0}^dC_i\Big)\int_{0}^{t_2}
\vert s-u\vert^{2H_0-3\over 2}\prod_{i=1}^d
\vert \widetilde{X}_{s}^i-x_i\vert^{2H_i-3\over 2}ds
$$
Therefore, the process $Z_t$ is sub-additive, which means that for any
$t_1,t_2>0$, $Z_{t_1+t_2}\le Z_{t_1}+\widetilde{Z}_{t_2}$, where
$\widetilde{Z}_{t_2}$ is independent of $\{Z_s;\hskip.1in s\le t_1\}$
and has a distribution same as $Z_{t_2}$.\\
\textbf{Step 3 }
Note that $Z_t\ge 0$ is non-decreasing and path-wise continuous.
By Theorem 1.3.5 in \cite{chen}, for any $\theta>0$ and $t>0$
$$
E\exp\big\{\theta Z_t\big\}<\infty
$$
and
$$
\lim_{t\to\infty}{1\over t}\log E\exp\big\{\theta Z_t\big\}=\Psi(\theta)
$$
where $0\le\Psi(\theta)<\infty$.
By the scaling property (\ref{scaling}), and the fact $Z_t=Y_t^{\frac12}$, we have  $Z_t\buildrel d\over =t^{\gamma/2}Z_1$
with \begin{equation}\label{gamma}
\gamma =\kappa+1=2H_0+{1\over\alpha}\sum_{i=1}^d(2H_i-2).
\end{equation}
Therefore
$$
\lim_{t\to\infty}{1\over t}\log E\exp\big\{\theta Z_t\big\}=\lim_{t\to\infty}{1\over {t\theta^\frac 2\gamma}}\log E\exp\big\{ Z_{t\theta^\frac 2\gamma}\big\}\theta^\frac 2\gamma=\Psi(1)\theta^\frac 2\gamma
$$

Using Chebyshev inequality, we have
 $$e^{\theta t}\P\big\{Z_t\ge t\big\}\le Ee^{\theta Z_t},$$
and hence  $$\theta t+\log\P\big\{Z_t\ge t\big\}\le \log Ee^{\theta Z_t}.$$
This implies
 $$\limsup_{t\to \infty}{1\over{t}} \log\P\big\{Z_t\ge t\big\}\le\lim_{t\to\infty} {1\over t} \log Ee^{\theta Z_t}-\theta=\Psi(\theta)-\theta=\theta^\frac 2\gamma\Psi(1)-\theta$$
and
$$\limsup_{t\to \infty}{1\over{t}} \log\P\big\{Z_t\ge t\big\}\le \min_{\theta>0}\left\{\theta^\frac 2\gamma\Psi(1)-\theta\right\},$$
where the right hand side is strictly negative when we choose $\theta$ sufficiently closed to $0$ (noting that $\gamma\in(1,2)$). Hence there exists $a>0$ such that
$$\limsup_{t\to \infty}\frac 1t\log\P\big\{Z_t\ge t\big\}\le -a.$$
Therefore there exists $N>0$, such that when $t>N$
$$
\P\Big\{Z_1\ge t^{2-\gamma\over 2}\Big\}\le \exp\{-\frac a2t\}.
$$

Since for a random variable $X\ge0,$ $$Ee^X=E\int_0^X e^ydy+1=\int_0^\infty \P\big\{X\ge y\}e^ydy+1,$$ we have

\begin{align*}
& E e^{\theta Z_1^\frac{2}{2-r}}\cr
=& \int_0^\infty \P\big\{\theta Z_1^\frac 2{2-r}\ge y\}e^ydy+1\cr
\le& \int_0^N e^ydy + \int_N^\infty \P\big\{\theta Z_1^\frac 2{2-r}\ge y\}e^ydy+1\cr
\le&\sum_{K=N}^{\infty}\P\big\{\theta Z_1^\frac 2{2-r}\ge K\}e^{K+1}+e^N\cr
\le&\sum_{K=N}^{\infty}e^{-\frac {aK} {2\theta} }e^{K+1}+e^N\cr
\end{align*}
This give the critical integrability
$$
E\exp\Big\{\theta Z_1^{2\over 2-\gamma}\Big\}<\infty,
$$
for  $\theta\in(0,\dfrac a2)$, which implies that $E\exp\Big\{\lambda Z_1^{2}\Big\}<\infty$ for all $\lambda>0$ since $1<\gamma<2.$

\end{proof}

\begin{proposition}\label{best}
Suppose $2H_{0}+\frac{1}{\alpha}\sum_{i=1}^d (2H_{i}-2)\le1$. Then, conditionally to $X$
the family  $V_{t,x}^{{\varepsilon },{\delta }}$ does not converge
in probability as $\varepsilon$ and $\delta$ tend to zero, for
a non-zero set of trajectories of $X$.
\end{proposition}

\begin{proof}
We will prove by contradiction. Suppose
$V_{t,x}^{{\varepsilon },{\delta }}$ converges to $V_{t,x}$
in probability as $\varepsilon$ and $\delta$ tend to zero given $X$ for
 almost all trajectories of $X$. Since given $X$, $V_{t,x}^{{\varepsilon },{\delta }}$ is Gaussian, so
$E^W (V_{t,x}^2)=\lim\limits_{\varepsilon\to0,\delta\to0}
E^W(V_{t,x}^{\varepsilon,\delta})^2=\int_0^t\int_0^t|s-r|^{2H_0-2}
\prod_{i=1}^d|X_s^i-X_r^i|^{2H_i-2}dsdr<\infty$ for almost
all trajectories of $X$. As in the Step 2 in the proof of the
Theorem \ref{exp}, we can prove that
$\left(\int_0^t\int_0^t|s-r|^{2H_0-2}\prod_{i=1}^d|X_s^i-X_r^i|^{2H_i-2}dsdr\right)^\frac12$
is sub-additive and it's also finite almost surely, hence we have
$$E\exp \left(\int_0^t\int_0^t|s-r|^{2H_0-2}\prod_{i=1}^d
|X_s^i-X_r^i|^{2H_i-2}dsdr\right)^\frac12<\infty,$$ implying that
$E\int_0^t\int_0^t|s-r|^{2H_0-2}\prod_{i=1}^d|X_s^i-X_r^i|^{2H_i-2}dsdr<\infty$
which is a contradiction.
\end{proof}

\setcounter{equation}{0}

\section{Feynman-Kac formula}

We recall that $W$ is a fractional Brownian sheet on
$\mathbb{R}_{+}\times
\mathbb{R}^{d}$ with Hurst parameters $(H_{0},H_{1},\dots ,H_{d})$ where $%
H_{i}\in (\frac{1}{2},1)$ for $i=0,\dots ,d$. For any $\varepsilon $, $%
\delta >0$ we define%
\begin{equation}  \label{pot}
\dot{W}^{\epsilon ,\delta
}(t,x):=\int_{0}^{t}\int_{\mathbb{R}^{d}}\varphi _{\delta
}(t-s)p_{\epsilon }(x-y)W(ds,dy).
\end{equation}%
%In order to give a notion of solution for the heat equation with
%fractional noise (\ref{e.1.1}) we need the following definition of
%the Stratonovitch integral, which is  equivalent to that of
%Russo-Vallois  in \cite{RV}.
We shall follow the analogous ideas  in \cite{hns} to define the Stratonovich integral  
with respect to the fractional noises and to define the concept of  solution to
the stochastic partial differential equations.  First we give the following two definitions which are  the same as    the  definitions 4.1 and 4.2 in \cite{hns}.  

\begin{definition}
\label{def2} Given a random field $v=\{v(t,x),t\geq 0,x\in
\mathbb{R}^d\}$ such that
\begin{equation*}
\int_{0}^{T}\int_{\mathbb{R}^d}|v(t,x)|dxdt<\infty
\end{equation*}
almost surely for all $T>0$, the Stratonovitch integral $\int_{0}^{T}\int_{%
\mathbb{R}^d}v(t,x)W(dt,dx)$ is defined as the following limit in
probability if it exists%
\begin{equation*}
\lim_{\epsilon ,\delta \downarrow 0}\int_{0}^{T}\int_{\mathbb{R}^d}v(t,x)%
\dot{W}^{\epsilon ,\delta }(t,x)dxdt.
\end{equation*}
\end{definition}

\begin{remark} When $W$ is the standard 
Brownian motion, similar definitions of Stratonovich integrals
were studied by a great number of authors (see  e.g. \cite{IW}, chapter 6.
See also \cite{HM} and the references therein for a more 
general approximations.)  For multiple parameter fractional Brownian motions,
our definition is exactly the same as in \cite{hns}.  If the temporal Hurst parameter
is less than $1/2$, we refer to \cite{hln} for a detailed study.  
\end{remark} 

We are going to consider the following notion of solution for Equation (\ref{e.1.1}).

\begin{definition}
\label{def} A random field $u=\{u(t,x),t\geq 0,x\in
\mathbb{R}^d\}$ is a
weak solution to Equation (\ref{e.1.1}) 
\begin{description}
\item{(i)}\quad For any $t>0$ and any compact domain $D$ of $\RR^d$,
\[
  \int_0^t \int_D \EE\left(|u(s, x)|\right) dsdx<\infty\,.
\]
\item{(ii)}\quad 
For any $C^{\infty }$ function $%
\varphi $ with compact support on $\mathbb{R}^d$, we have
\begin{eqnarray*}
\int_{\mathbb{R}^d}u(t,x)\varphi
(x)dx&=&\int_{\mathbb{R}^d}f(x)\varphi (x)dx - \frac 12
\int_{0}^{t}\int_{\mathbb{R}^d}u(s,x) (-\Delta)^\frac{\alpha}{2}
\varphi (x)dxds
\\
&& +\int_{0}^{t}\int_{\mathbb{R}^d}u(s,x)\varphi (x)W(ds,dx),
\end{eqnarray*}%
almost surely, for all $t\ge 0$, where the last term is a
Stratonovitch stochastic integral in the sense of Definition
\ref{def2}.
\end{description}
\end{definition}

The following is the main result of this section.

\begin{theorem}  \label{th1}
Suppose that $2H_0+\frac{1}{\alpha} \sum_{i=1}^d (2H_{i}-2)>1$ and
that $f$ is a bounded
measurable function. Then process%
\begin{equation}  \label{utx}
u(t,x)=E^{X}\left( f(X_{t}^{x})\exp \left( \int_{0}^{t}\int_{\mathbb{R}%
^d}\delta (X_{t-r}^{x}-y)W(dr,dy)\right) \right)
\end{equation}%
is a weak solution to Equation (\ref{e.1.1}).
\end{theorem}

\begin{proof}
Consider the approximation of the Equation (\ref{e.1.1}) given by
the following heat equation with a random potential
\begin{equation}
\begin{cases}
\dfrac{\partial u^{\varepsilon ,\delta }}{\partial t}= -(-\Delta)^{\frac \alpha2}u^{\varepsilon ,\delta }+u^{\varepsilon ,\delta
}\dot{W}_{t,x}^{\epsilon
,\delta } \\
u^{\varepsilon ,\delta }(0,x)=f(x).%
\end{cases}
\label{approx}
\end{equation}
From the classical Feynman-Kac formula (see for example \cite{A}, Theorem 
6.7.9) we know that%
\begin{equation*}
u^{\varepsilon ,\delta }(t,x) =E^{X}\left( f(X_{t}^{x})\exp \left(
\int_{0}^{t}\dot{W}^{\epsilon ,\delta }(t-s,X_{s}^{x})ds\right)
\right),
\end{equation*}%
where $X^x$ is a $d$-dimensional symmetric $\alpha$-stable L\'evy motion
independent of $W$
starting at $x$. By Fubini's theorem we can write%
\begin{eqnarray*}
\int_{0}^{t}\dot{W}^{\varepsilon ,\delta }(t-s,X_{s}^{x})ds
&=&\int_{0}^{t}\left( \int_{0}^{t}\int_{\mathbb{R}^d}\varphi
_{\delta
}(t-s-r)p_{\varepsilon }(X_{s}^{x}-y)W(dr,dy)\right) ds \\
&=&\int_{0}^{t} \int_{\mathbb{R}^d} \left( \int_{0}^{t}\varphi
_{\delta
}(t-s-r)p_{\varepsilon }(X_{s}^{x}-y)ds\right) W(dr,dy) \\
&=&V_{t,x}^{\varepsilon ,\delta },
\end{eqnarray*}%
where $V_{t,x}^{\varepsilon ,\delta }$ is defined in (\ref{e3}). Therefore,%
\begin{equation*}
u ^{\varepsilon ,\delta }(t,x)=E^{X}\left( f(X_{t}^{x})\exp \left(
V_{t,x}^{\varepsilon ,\delta }\right) \right) .
\end{equation*}

\medskip \noindent \textbf{Step 1 }We will prove that for any $x\in \mathbb{R%
}^d$ and any $t>0$, we have%
\begin{equation}
\lim_{\varepsilon ,\delta \downarrow 0}\ \ E^W |u ^{\varepsilon
,\delta }(t,x)-u(t,x)|^{p}=0,  \label{a1}
\end{equation}%
for all $p\geq 2$, where $u(t,x)$ is defined in (\ref{utx}). Notice that%
\begin{eqnarray*}
E^W |u ^{\varepsilon ,\delta }(t,x)-u(t,x)|^{p} &=&E^{W}\left|
E^{B}\left( f(B_{t}^{x})\left[ \exp \left( V_{t,x}^{\varepsilon
,\delta }\right) -\exp
\left( V_{t,x}\right) \right] \right) \right| ^{p} \\
&\leq &\left\| f\right\| _{\infty }^{p}E\left| \exp \left(
V_{t,x}^{\varepsilon ,\delta }\right) -\exp \left( V_{t,x}\right)
\right| ^{p},
\end{eqnarray*}%
where $V_{t,x}$ is defined in (\ref{e.3.4}). Since $\exp \left(
V_{t,x}^{\varepsilon ,\delta }\right) $ converges to $\exp \left(
V_{t,x}\right) $ in probability by Theorem \ref{teo1}, to show
(\ref{a1}) it suffices to prove that for any $\lambda \in
\mathbb{R}$
\begin{equation}
\sup_{\epsilon ,\delta }E\exp \left( \lambda V_{t,x}^{\varepsilon
,\delta }\right) <\infty.  \label{a2}
\end{equation}%
The estimate (\ref{a2}) follows from (\ref{e3}), (\ref{e5}), and (\ref{e21}%
):
\begin{eqnarray}
&& E \exp \left( \lambda V_{t,x}^{\varepsilon ,\delta }\right) = E
\exp \left( \frac{\lambda ^{2}}{2}\left\| A_{t,x}^{\varepsilon
,\delta }\right\|
_{\mathcal{H}}^{2}\right)  \notag \\
&& \qquad \leq E \exp \left( \frac{\lambda ^{2}}{2}C\int_{0}^{t}%
\int_{0}^{t}|r -s |^{2H_{0}-2} \prod_{i=1}^d\left| X^i_{r}-X^i_{s
}\right|
^{2H_{i}-2}dr ds\right)  \notag \\
&&\qquad <\infty.  \label{a2a}
\end{eqnarray}

\medskip \noindent \textbf{Step 2 } \ Now we prove that $u(t,x)$ is a weak
solution to Equation (\ref{e.1.1}) in the sense of Definition
\ref{def}. Suppose $\varphi $ is a smooth function with compact
support. We know that,
\begin{eqnarray}
\int_{\mathbb{R}^d}u^{\varepsilon ,\delta }(t,x)\varphi (x)dx &=&\int_{%
\mathbb{R}^d}f(x)\varphi (x)dx- \frac 12
\int_{0}^{t}\int_{\mathbb{R}^d}u
^{\varepsilon ,\delta } (s,x) (-\Delta)^\frac{\alpha}{2} \varphi (x)dxds  \notag \\
&&+\int_{0}^{t}\int_{\mathbb{R}^d}u ^{\varepsilon ,\delta }(s,x)\varphi (x)%
\dot{W}^{\varepsilon ,\delta }(s,x)dsdx.  \label{a3}
\end{eqnarray}
From Step 1, we know that $\sup\limits_{\varepsilon,\delta}\sup\limits_{s\in[0,T], x\in\R^d}E|u^{\varepsilon,\delta}(s,x)-u(s,x)|<\infty$. Then by dominant convergence theorem, we have
\[\lim_{\varepsilon,\delta}\int_0^t\int_{\R^d}u^{\varepsilon,\delta}(s,x)(-\Delta)^{\frac\alpha2}\varphi(x)dx=\int_0^t\int_{\R^d}u(s,x)(-\Delta)^{\frac\alpha2}\varphi(x)dx.\]
Therefore, it suffices to prove that
\begin{equation*}
\lim_{\varepsilon ,\delta \downarrow
0}\int_{0}^{t}\int_{\mathbb{R}^d}u ^{\varepsilon ,\delta
}(s,x)\varphi (x)\dot{W}^{\varepsilon ,\delta
}(s,x)dsdx=\int_{0}^{t}\int_{\mathbb{R}^d}u(s,x)\varphi
(x)W(ds,dx),
\end{equation*}
in probability. From (\ref{a3}) and (\ref{a1}) it follows that $
\int_{0}^{t}\int_{\mathbb{R}^d}u ^{\epsilon ,\delta }(s,x)\varphi (x)\dot{W}%
^{\epsilon ,\delta }(s,x)dsdx$ converges in $L^{2}$ to the random variable%
\begin{equation*}
G=\int_{\mathbb{R}^d}u(t,x)\varphi
(x)dx-\int_{\mathbb{R}^d}f(x)\varphi (x)dx+ \frac 12
\int_{0}^{t}\int_{\mathbb{R}^d}u(t,x)(-\Delta)^\frac{\alpha}{2}
\varphi (x)dxds
\end{equation*}
as $\varepsilon $ and $\delta$ tend to zero. Hence, if
\begin{equation*}
B_{\varepsilon ,\delta }=\int_{0}^{t}\int_{\mathbb{R}^d}(u
^{\epsilon ,\delta }(s,x)-u(s,x))\varphi (x)\dot{W}^{\varepsilon
,\delta }(s,x)dsdx
\end{equation*}%
converges in $L^{2}$ to zero, then
\begin{equation*}
\int_{0}^{t}\int_{\mathbb{R}^d}u(s,x)\varphi (x)\dot W^{{\varepsilon}, {%
\delta}}dsdx = \int_{0}^{t}\int_{\mathbb{R}^d}u^{{\varepsilon},
{\delta}}
(s,x)\varphi (x)\dot W^{{\varepsilon}, {\delta}}dsdx -B_{{\varepsilon}, {%
\delta}}
\end{equation*}%
converges to $G$ in $L^2$. Thus $u(s,x)\varphi (x)$ will be
Stratonovich integrable and
\begin{equation*}
\int_{0}^{t}\int_{\mathbb{R}^d}u(s,x)\varphi (x)W(ds,dx)=G\,,
\end{equation*}%
which will complete the proof. In order to show the convergence to zero of $
B_{ \varepsilon ,\delta },$ noting that both $u^{\varepsilon, \delta}(s,x)$ and $u(s,x)$ belong to   $\mathbb D^{1,2}$, we can express the product $(u
^{\varepsilon ,\delta }(s,x)-u(s,x))\dot{W}^{\varepsilon ,\delta
}(s,x)$ as the sum of a divergence integral plus a trace term (see
(\ref{a5})),
\begin{eqnarray*}
&& (u ^{\varepsilon ,\delta }(s,x)-u(s,x))\dot{W}^{\varepsilon
,\delta }(s,x)
\\
&& \quad = \int_{0}^{t}\int_{\mathbb{R}^d}(u ^{\varepsilon ,\delta
}(s,x)-u(s,x))\varphi _{\delta }(s-r)p_{\varepsilon }(x-z)\delta W_{r,z} \\
&&\qquad +\langle D(u ^{\varepsilon ,\delta }(s,x)-u(s,x)),\varphi
_{\delta }(s-\cdot )p_{\varepsilon }(x-\cdot )\rangle
_{\mathcal{H}}\,.
\end{eqnarray*}%
Then we have
\begin{eqnarray}
B_{\varepsilon ,\delta } &=& \int_{0}^{t}\int_{\mathbb{R}^d}\phi
_{r,z}^{\epsilon ,\delta }\delta W_{r,z}  \notag \\
&& +\int_{0}^{t}\int_{\mathbb{R}^d}\varphi (x)\langle D(u
^{\varepsilon ,\delta }(s,x)-u(s,x)),\varphi _{\delta }(s-\cdot
)p_{\varepsilon }(x-\cdot
)\rangle _{\mathcal{H}}dsdx  \notag \\
&=&B_{\varepsilon ,\delta }^{1}+B_{\varepsilon ,\delta }^{2},
\label{BB}
\end{eqnarray}%
where
\begin{equation*}
\phi _{r,z}^{\varepsilon ,\delta
}=\int_{0}^{t}\int_{\mathbb{R}^d}(u ^{\varepsilon ,\delta
}(s,x)-u(s,x))\varphi (x)\varphi _{\delta }(s-r)p_{\varepsilon
}(x-z)dsdx,
\end{equation*}%
and $\delta (\phi ^{\varepsilon ,\delta })= \int_{0}^{t}\int_{\mathbb{R}%
^d}\phi _{r,z}^{\varepsilon ,\delta }\delta W_{r,z}$ denotes the
divergence or the Skorohod integral of $\phi^{\varepsilon ,\delta
}$.

\medskip \noindent \textbf{Step 3 } \  For the term $B_{\varepsilon ,\delta
}^{1}$ we use the following $L^{2}$ estimate for the Skorohod
integral (see Proposition 1.3.1 in \cite{nualart})
\begin{equation}
E[(B_{\varepsilon ,\delta }^{1})^{2}]\leq E(\Vert \phi
^{\varepsilon ,\delta
}\Vert _{\mathcal{H}}^{2})+E(\Vert D\phi ^{\varepsilon ,\delta }\Vert _{%
\mathcal{H}\otimes \mathcal{H}}^{2})\,.  \label{e.4.2}
\end{equation}
The first term in (\ref{e.4.2}) is estimated as follows%
\begin{eqnarray}
E(\Vert \phi ^{\varepsilon ,\delta }\Vert _{\mathcal{H}}^{2}) &=&
\int_{0}^{t}\int_{\mathbb{R}^d}\int_{0}^{t}\int_{\mathbb{R}^d}E\left[
(u
^{\varepsilon ,\delta }(s,x)-u(s,x))(u ^{\varepsilon ,\delta }(r,y)-u(r,y))%
\right]  \notag \\
&& \!\!\!\!\!\!\!\!\!\!\!\!\!\!\!\!\!\!\!\!\!\!\times \varphi
(x)\varphi (y)\langle \varphi _{\delta }(s-\cdot )p_{\varepsilon
}(x-\cdot ),\varphi _{\delta }(r-\cdot )p_{\varepsilon }(y-\cdot
)\rangle _{\mathcal{H}}dsdxdrdy. \label{a7}
\end{eqnarray}%
Using  Lemmas \ref{lemma2} and \ref{lemma3} we can write%
\begin{eqnarray}
&&\langle \varphi _{\delta }(s-\cdot )p_{\epsilon }(x-\cdot
),\varphi _{\delta }(r-\cdot )p_{\varepsilon }(y-\cdot )\rangle
_{\mathcal{H}}  \notag
\\
&=&\alpha _H\left( \int_{[0,t]^{2}}\varphi _{\delta
}(s-\sigma)\varphi _{\delta }(r- \tau )|\sigma- \tau |^{2H_{0}-2}d
\sigma d\tau \right)  \notag
\\
&&\times \left( \int_{\mathbb{R}^{2d}}p_{\varepsilon }(x-z
)p_{\varepsilon
}(y-w) \prod_{i=1}^d |z_i -w_i|^{2H_{i}-2}dz dw\right)  \notag \\
&\leq &C|s-r|^{2H_{0}-2} \prod_{i=1}^d |x-y|^{2H_{i}-2},
\label{a8}
\end{eqnarray}%
for some constant $C>0$. As a consequence, the integrand on the
right-hand side of Equation (\ref{a7}) converges to zero as
$\varepsilon $ and $\delta $ tend to zero for any $s$, $r$, $x$,
$y$ due to (\ref{a1}). From (\ref{a2a})
we get%
\begin{equation}
\sup_{\varepsilon ,\delta }\sup_{x\in \mathbb{R}^d}\sup_{0\leq
s\leq t}E\left( u ^{\epsilon ,\delta }(s,x)\right) ^{2}\leq
\left\| f\right\| _{\infty }^{2}\sup_{\varepsilon ,\delta
}\sup_{x\in \mathbb{R}^d}\sup_{0\leq s\leq t}E\exp \left(
2V_{s,x}^{\varepsilon ,\delta }\right) <\infty . \label{a9}
\end{equation}%
Hence, from (\ref{a8}) and (\ref{a9}) we get that the integrand on
the
right-hand side of Equation (\ref{a7}) is bounded by $C|s-r|^{2H_{0}-2}%
\prod_{i=1}^d |x_i-y_i|^{2H_{i}-2}$, for some constant $C>0$.
Therefore, by dominated convergence we get that $E(\Vert \phi
^{\varepsilon ,\delta }\Vert _{\mathcal{H}}^{2})$ converges to
zero as $\varepsilon $ and $\delta $ tend to zero.

\medskip \noindent \textbf{Step 4 } \ On the other hand, we have
\begin{equation*}
D(u ^{\varepsilon ,\delta }(t,x))=E^{X}\left[ f(X^x_t)\exp
(V_{t,x}^{\varepsilon ,\delta })A_{t,x}^{\varepsilon ,\delta
}\right] ,
\end{equation*}%
where $A_{t,x}^{\varepsilon ,\delta }$ is defined in \ (\ref{e2}). Therefore,%
\begin{eqnarray}
&&E\langle D(u ^{\varepsilon ,\delta }(t,x)),D(u ^{\varepsilon
^{\prime
},\delta ^{\prime }}(t,x))\rangle _{\mathcal{H}}  \notag \\
&=&E^{W}E^{X}\Big( f(X_{t}^{1}+x)f(X_{t}^{2}+x)  \notag \\
&& \times \exp (V_{t,x}^{\varepsilon ,\delta
}(X^{1})+V_{t,x}^{\varepsilon ,\delta }(X^{2}))\langle
A_{t,x}^{\varepsilon ,\delta
}(X^{1}),A_{t,x}^{\varepsilon ^{\prime },\delta ^{\prime }}(X^{2})\rangle _{%
\mathcal{H}}\Big) ,  \label{e.4.12a}
\end{eqnarray}
where $X^{1}$ and $X^{2}$ are two independent $d$-dimensional
symmetric $\alpha$-stable L\'evy motions, and here $E^X$ denotes the expectation with respect to $(X^1, X^2)$%
. Then from the previous results it is easy to show that%
\begin{eqnarray}
&&\lim_{\varepsilon ,\delta \downarrow 0}E\langle D(u
^{\varepsilon ,\delta
}(t,x)),D(u ^{\varepsilon ^{\prime },\delta ^{\prime }}(t,x))\rangle _{%
\mathcal{H}}  \notag \\
&=&E \left[ f(X_{t}^{1}+x)f(X_{t}^{2}+x)\right.  \notag \\
&&\times \exp \left( \dfrac{\alpha _H}{2} \sum_{j,k=1}^{2}\int_{0}^{t}%
\int_{0}^{t}|s-r|^{2H_{0}-2} \prod_{i=1}^d
|X_{s}^{j,i}-X_{r}^{k,i}|^{2H_{i}-2}dsdr\right)  \notag \\
&&\left. \times \alpha _H \int_{0}^{t}\int_{0}^{t}|s-r|^{2H_{0}-2}
\prod_{i=1}^d |X_{s}^{1,i}-X_{r}^{2,i}|^{2H_{i}-2}dsdr\right] .
\label{b1}
\end{eqnarray}
This implies that $u ^{\varepsilon ,\delta }(t,x)$ converges in the space $%
\mathbb{D}^{1,2}$ to $u(t,x)$ as $\delta \downarrow 0$ and
$\varepsilon
\downarrow 0$. Letting ${\varepsilon}^{\prime}={\varepsilon}$ and ${\delta}%
^{\prime}={\delta}$ in \hbox{(\ref{e.4.12a})} and using the same
argument as for \hbox{(\ref{a9})}, we obtain
\begin{equation*}
\sup_{\varepsilon ,\delta }\sup_{x\in \mathbb{R}^d}\sup_{0\leq
s\leq
t}E\left\| D(u ^{\varepsilon ,\delta }(s,x))\right\| _{\mathcal{H}%
}^{2}<\infty .
\end{equation*}
Then
\begin{eqnarray*}
E\Vert D\phi ^{\varepsilon ,\delta }\Vert _{\mathcal{H}\otimes \mathcal{H}%
}^{2} &=&\int_{0}^{t}\int_{\mathbb{R}^d}\int_{0}^{t}\int_{\mathbb{R}%
}E\langle D(u ^{\varepsilon ,\delta }(s,x)-u(s,x)),D(u
^{\varepsilon ,\delta
}(r,y)-u(r,y))\rangle _{\mathcal{H}} \\
&&\times \varphi (x)\varphi (y)\langle \varphi _{\delta }(s-\cdot
)p_{\varepsilon }(x-\cdot ),\varphi _{\delta }(r-\cdot
)p_{\varepsilon }(y-\cdot )\rangle _{\mathcal{H}}dsdxdrdy\
\end{eqnarray*}%
converges to zero as $\varepsilon $ and $\delta $ tend to zero. \ Hence, by (%
\ref{e.4.2}) $B_{\varepsilon ,\delta }^{1}$ converges to \ zero in
$L^{2}$ as $\varepsilon $ and $\delta $ tend to zero.

\medskip \noindent \textbf{Step 5} \ The second summand in the right-hand
side of (\ref{BB}) \ can be written as%
\begin{eqnarray*}
B_{\varepsilon ,\delta }^{2}
&=&\int_{0}^{t}\int_{\mathbb{R}^d}\varphi (x)\langle D(u
^{\varepsilon ,\delta }(s,x)-u(s,x)),\varphi _{\delta
}(s-\cdot )p_{\varepsilon }(x-\cdot )\rangle _{\mathcal{H}}dsdx \\
&=&\int_{0}^{t}\int_{\mathbb{R}^d}\varphi (x)E^{X}\left(
f(X_{s}^{x})\exp \left( V_{s,x}^{\varepsilon ,\delta }\right)
\langle A_{s,x}^{\varepsilon
,\delta },\varphi _{\delta }(s-\cdot )p_{\epsilon }(x-\cdot )\rangle _{%
\mathcal{H}}\right) dsdx \\
&&-\int_{0}^{t}\int_{\mathbb{R}}\varphi (x)E^{X}\left(
f(X_{s}^{x})\exp \left( V_{s,x}\right) \langle \delta (X_{s-\cdot
}^{x}-\cdot ),\varphi _{\delta }(s-\cdot )p_{\varepsilon }(x-\cdot
)\rangle _{\mathcal{H}}\right)
dsdx \\
&=&B_{\varepsilon ,\delta }^{3}-B_{\varepsilon ,\delta }^{4}
\end{eqnarray*}
where%
\begin{eqnarray*}
\langle A_{s,x}^{\varepsilon ,\delta },\varphi _{\delta }(s-\cdot
)p_{\epsilon }(x-\cdot )\rangle _{\mathcal{H}} &=&\ \alpha
_H\int_{[0,s]^{3}}\int_{\mathbb{R}^{2d}}\ |r -v|^{2H_{0}-2}
\prod_{i=1}^d
|y_i-z_i|^{2H_{i}-2} \\
&&\times \varphi _{\delta }(s-r )p_{\varepsilon }(X_{r}^x-y ) \\
&&\times \varphi _{\delta }(s-v)p_{\varepsilon }(x-z)dy dyzdrdrdv,
\end{eqnarray*}%
and%
\begin{eqnarray*}
&&\langle \delta (X_{s-\cdot }^{x}-\cdot ),\varphi _{\delta
}(s-\cdot
)p_{\varepsilon }(x-\cdot )\rangle _{\mathcal{H}} \\
&=&\alpha _H \int_{[0,s]^{2}}\int_{\mathbb{R}^d}\ v^{2H_{0}-2}
\prod_{i=1}^d |X_{r} ^{x_i} -y_i|^{2H_{i}-2}\varphi _{\delta
}(r-v)p_{\varepsilon }(x-y)dydv dr.
\end{eqnarray*}%
Lemma \ref{lemma2} and Lemma \ref{lemma3} \ imply that%
\begin{equation}
\langle A_{s,x}^{\varepsilon ,\delta },\varphi _{\delta }(s-\cdot
)p_{\varepsilon }(x-\cdot )\rangle _{\mathcal{H}}\leq
C\int_{0}^{s}r^{2H_{0}-2} \prod_{i=1}^d |X^i_{r}|^{2H_{i}-2}dr,
\label{c1}
\end{equation}%
and%
\begin{equation}
\langle \delta (B_{s-\cdot }^{x}-\cdot ),\varphi _{\delta
}(s-\cdot )p_{\varepsilon }(x-\cdot )\rangle _{\mathcal{H}}\leq
C\int_{0}^{s}r^{2H_{0}-2} \prod_{i=1}^d |X^i_{r}|^{2H_{i}-2}dr,
\label{c2}
\end{equation}%
for some constant $C>0$. Then, from (\ref{c1}) and (\ref{c2}) and
from the fact that the random variable $\int_{0}^{s}r^{2H_{0}-2}
\prod_{i=1}^d |X^i_{r}|^{2H_{i}-2}dr$ is square integrable because
of Lemma \ref{lemma4}, we can apply the dominated convergence
theorem and get that $B_{\varepsilon
,\delta }^{3}$ and $B_{\varepsilon ,\delta }^{4}$ converge both in $L^{2}$ to%
\begin{equation*}
\alpha _H\int_{0}^{t}\int_{\mathbb{R}^d}\varphi (x)E^{B}\left(
f(X_{s}^{x})\exp \left( V_{s,x}\right) \int_{0}^{s}r^{2H_{0}-2}
\prod_{i=1}^d |X^i_{r}|^{2H_{i}-2}dr\right) dsdx,
\end{equation*}%
as $\varepsilon $ and $\delta $ tend to zero. Therefore \
$B_{\varepsilon ,\delta }^{2}$ converges in $L^{2}$ to zero as
$\varepsilon $ and $\delta $ tend to zero. This completes the
proof.
\end{proof}

\begin{remark}
 The uniqueness of the solution remains to be investigated in a
future work. The definition of the Stratonovich integral as a
limit in probability makes the uniqueness problem nontrivial, and
it is not clear how to proceed.
\end{remark}

As a corollary  of Theorem \ref{th1} we obtain the following
result.

\begin{corollary}
\label{t.4.4} Suppose $2H_{0}+\frac1\alpha\sum_{i=1}^{d}(2H_{i}-2)>1$. Then the solution $%
u(t,x)$ given by \hbox{(\ref{utx})} has finite moments of all
orders. Moreover, for any positive integer $p$, we have
\begin{eqnarray}
&&E\left( u(t,x)^{p}\right) =E\Bigg(\prod_{j=1}^{p}f(X_{t}^{j}+x)
\label{e.4.7} \\
&&\times \exp \left[ \frac{\alpha _{H}}{2}\sum_{j,k=1}^{p}\int_{0}^{t}%
\int_{0}^{t}|s-r|^{2H_{0}-2}%
\prod_{i=1}^{d}|X_{s}^{j,i}-X_{r}^{k,i}|^{2H_{i}-2}dsdr\right]
\Bigg)\,, \nonumber
\end{eqnarray}%
where $X_{1},\ldots ,X_{p}$ are independent $d$-dimensional
standard symmetric $\alpha$-stable L\'evy motions.
\end{corollary}
\begin{proof}
Using the Feynman-Kac formula (\ref{utx}), we have
\begin{align*}
&E(u(t,x)^p)\\
=&E^W\left[\prod_{j=1}^p E^{X_j} \left(f(X_t^j+x)\exp\left(\int_0^t\int_{\R^d}\delta(X_{t-r}^j+x-y)W(dr,dy)\right)\right)\right]\\
=&E^{X_1,\cdots,X_p}\left[\prod_{j=1}^p f(X_t^j+x) E^W \exp\left(\sum_{j=1}^p\int_0^t\int_{\R^d}\delta(X_{t-r}^j+x-y)W(dr,dy)\right)\right] \\
=&E^{X_1,\cdots,X_p}\left[\prod_{j=1}^p f(X_t^j+x) \exp\left(\frac12E^W\left(\int_0^t\int_{\R^d}\sum_{j=1}^p\delta(X_{t-r}^j+x-y)W(dr,dy)\right)^2\right)\right] \\
=&E\left[\prod_{j=1}^p f(X_t^j+x) \exp\left(\frac{\alpha_H}2\sum_{j,k=1}^p\int_0^t\int_0^t|s-r|^{2H_0-2}\prod_{i=1}^d
|X_s^{j,i}-X_r^{k,i}|^{2H_i-2}dsdr\right)\right],
\end{align*}
where the last equation follows form \eref{e.2.14}. 
\end{proof}

\setcounter{equation}{0}
\section {H\"{o}lder continuity of the solution}

In this section, we study the H\"older continuity of the
solution to Equation (\ref{e.1.1}). The main result of this
section is the following theorem.

\begin{theorem}
\label{t.6.1} Suppose that $2H_0+\frac{1}{\alpha}\sum_{i=1}^d
(2H_{i}-2)>1$ and let $u(t,x)$ be a weak solution of Equation
(\ref{e.1.1}) given by (\ref{utx}). Then $u(t,x)$ has a continuous modification such
that for any $\rho \in \left(0, \frac \kappa 2 \right)$ (where
$\kappa$ has been defined in (\ref{kappa})), and any compact
rectangle $I \subset \mathbb{R}_+\times\mathbb{R}^d $ there exists
a
positive random variable $K_I $ such that almost surely, for any $%
(s,x),(t,y)\in I$ we have
\begin{equation*}
| u(t,y )- u(s,x )|\le K_I (|t-s|^\rho +|y-x|^{\alpha\rho }).
\end{equation*}
\end{theorem}

\begin{proof}
The proof will be done in several steps.

\medskip \noindent \textbf{Step 1} \quad Recall that $V_{t,x}=\int_0^t\int_{%
\mathbb{R}^d} \delta(X_{t-r}^x-y)W(dr,dy)$ denotes the random
variable introduced in (\ref{e.3.4}) and
\begin{equation*}
u(t,x)= E^X \left(f(X^x_t) \exp \left(V(_{t,x} \right) \right).
\end{equation*}
Set $V =V_{s,x}$ and $\tilde V=V _{t,y }$. Then we can write
\begin{eqnarray*}
&& E^W|u(s,x )-u(t,y)|^p = E^W|E^X (e^{V }-e^{\tilde V})|^p \\
&& \qquad \le E^W \big(E^X [|\tilde V -V |e^{ \max(V , \tilde V )}]\big)^p \\
&& \qquad\le E^W \big[ \big(E^Xe^{2 \max(V , \tilde V ) }\big)^{p/2}\big(%
E^X(\tilde V -V )^2\big)^{p/2}\big] \\
&& \qquad \le [E^W E^X e^{2p \max(V , \tilde V ) }\big]^\frac{1}{2} \big[E^W%
\big(E^X (\tilde V -V )^2\big)^p\big]^\frac{1}{2}.
\end{eqnarray*}
Applying Minkowski's inequality, the equivalence between the $L^2$
norm and the $L^p$ norm for a Gaussian random variable, and using
the exponential integrability property (\ref{e21}) we obtain
\begin{eqnarray}
&& E^W|u(s,x)-u(t,y)|^p \le C \big[E^W\big(E^X (\tilde V -V )^2\big)^p\big]^%
\frac{1}{2}  \notag \\
&& \quad\le C_p \big[E^X E^W|\tilde V -V |^2\big]^{p/2}.
\label{j1}
\end{eqnarray}
In a similar way to (\ref{e.2.14}) we can deduce the following
formula for the conditional variance of $\tilde V -V $
\begin{eqnarray}
&&E^W|\tilde V -V |^2= \alpha_H E^X\Bigg( \int_0^{s}\int_0^{s} |r
-v|^{2H_0-2} \prod_{i=1}^d |X^{i}_{s-r }-X^{i}_{s-v}|^{2H_i-2}drdv  \notag \\
&& \quad + \int_0^{t }\int_0^{t } |r-v|^{2H_0-2}
\prod_{i=1}^d|X^{i}_{t -r
}-X^{i}_{t -v}|^{2H_i-2}drdv  \notag \\
&&\quad -2\int_0^{s}\int_0^{t } |r-v|^{2H_0-2}
\prod_{i=1}^d|X_{s-r
}^{i}-X_{t-v}^{i} +x_i -y_i|^{2H_i-2}drdv\Bigg)  \notag \\
&&\quad := \alpha_H C(s, t , x , y).  \label{j2}
\end{eqnarray}

\medskip \noindent \textbf{Step 2} $\quad $ Fix $1\leq j\leq d$. Let us
estimate $C(s,t,x,y)$ when $s=t,$ and $x_{i}=y_{i}$ for all
$i\not=j$. We can write
\begin{equation}
C(t,t,x,y)=2\int_{0}^{t}\int_{0}^{t}|r-v|^{\kappa -1}\prod_{i\not=j}^{d}E%
\big(|\xi |^{2H_{i}-2}\big)E\big(|\xi |^{2H_{j}-2}-|z+\xi |^{2H_{j}-2}\big)%
drdv,  \label{eqa1}
\end{equation}%
where $z=\dfrac{x_{j}-y_{j}}{|r-v|^\frac 1\alpha}$ and $\xi $ is a
standard $\alpha$-stable
variable.  By Lemma \ref{Lem7.6} the factor $E%
\big(|\xi |^{2H_{j}-2}-|z+\xi |^{2H_{j}-2}\big)$ can be bounded by
a constant if $|r-v|\leq (x_{j}-y_{j})^{\alpha}$, and it can be bounded by $%
C|x_{j}-y_{j}|^2|r-v|^{-\frac 2\alpha}$ if $%
|r-s|>(x_{j}-y_{j})^{\alpha} $. In this way we obtain
\begin{eqnarray*}
&&C(t,t,x,y)\leq C\int_{\{0<r,v<t,|r-v|\leq
(x_{j}-y_{j})^{\alpha}\}}|r-v|^{\kappa -1}drdv \\
&&\qquad
+C|x_{j}-y_{j}|^{2}\int_{\{0<r,v<t,|r-v|>(x_{j}-y_{j})^{\alpha}\}}|r-v|^{\kappa
-1-\frac 2\alpha}drdv \\
&&\qquad \qquad \qquad \leq C|x_{j}-y_{j}|^{\alpha\kappa }.
\end{eqnarray*}%
So, from (\ref{j1}) we have
\begin{equation}
E^{W}|u(t,x)-u(t,y)|^{p}\leq C|x_{j}-y_{j}|^{\frac\alpha 2 \kappa p}.
\label{e.6.1}
\end{equation}

\medskip \noindent \textbf{Step 3} \quad Suppose now that $s<t ,$ and $x =y$%
. Set $\delta=\frac{1}{\alpha}\sum_{i=1}^d (2H_i-2)$. We have
\begin{align*}
&C(s,t,x,x) \\
=&C\Bigg[ \int_{s} ^{t } \int_{s} ^{t } |r-v |^{\kappa-1 }drdv \\
&+\int_0^{s}\int_0^{t }|r-v|^{2H_0-2}\big(|r -v|^{\delta}-|r -v+
t-s|^{\delta }\big)drdv\Bigg] \,.
\end{align*}
The first integral is $O((t-s)^{\kappa +1})$, when $t-s$ is small.
For the second integral we use the change of variable $\sigma =v-r, v=\tau $, and we have
\begin{align*}
&\int_0^{s}\int_0^{t }|r-v|^{2H_0-2}\big(|r -v|^{\delta}-|r -v+
t-s|^{\delta
}\big)drdv \\
\le & \int_0^{s } d\tau \int_{-t }^{s} |\sigma|^{2H_0-2}\big|
|\sigma
|^{\delta }-|\sigma + t-s |^{\delta}\big|d\sigma \\
=& t \bigg[ \int_0^{s} \sigma ^{2H_0-2}\big(\sigma
^{\delta}-(\sigma +
t-s)^{\delta}\big)d\sigma \\
&+\int_{-t}^{s-t} (-\sigma)^{2H_0-2}\big((-\sigma-
t+s)^{\delta}-(-\sigma)^{\delta}\big)d\sigma \\
&+\int_{s-t}^0 (-\sigma)^{2H_0-2}\big|(-\sigma
)^{\delta}-(\sigma+t-s)^{\delta }\big|d\sigma \bigg] \\
=& t [A^{\prime}+B^{\prime}+C^{\prime}].
\end{align*}
For the first term in the above decomposition we can write
\begin{align*}
A^{\prime}&=(t-s)^{\kappa } \int_0^{\frac{s}{t-s}} \sigma ^{2H_0-2}\big(%
\sigma ^{\delta}-(\sigma+1)^{\delta}\big)d\sigma \\
& \le (t-s)^{\kappa } \int_0^\infty \sigma^{2H_0-2}\big(%
\sigma^{\delta}-(\sigma +1)^{\delta}\big)d\sigma \\
& \le C(t-s)^{\kappa },
\end{align*}
because $2H_0-2 + \delta-1<-1.$ Similarly we can get that
\begin{equation*}
B^{\prime}\le (t-s)^{\kappa } \int_1^\infty \sigma
^{2H_0-2}\big(\sigma ^{\delta }-(\sigma +1)^{\delta}\big)d\sigma.
\end{equation*}
At last,
\begin{equation*}
C^{\prime}\le \int_0^{t-s} \sigma ^{2H_0-2}\big(\sigma^{\delta}+(
t-s-\sigma)^{\delta }\big)d\sigma =C (t-s)^{\kappa }.
\end{equation*}
So we have
\begin{equation}  \label{e.6.2}
E^W|u(s,x)-u(t,y)|^p\le C (t-s)^{\frac \kappa 2 p}.
\end{equation}

\medskip \noindent \textbf{Step 4} \quad Combining Equation \ref{e.6.1} and
Equation \ref{e.6.2} with the estimates (\ref{j1}) and (\ref{j2}),
the
result of this theorem now can be concluded from Kolmogorov continuity criterion.
\end{proof}

\setcounter{equation}{0} 
\section{Skorohod type equations and chaos expansion}

In this section we consider the following heat equation on
$\mathbb{R}^d$
\begin{equation}
\begin{cases}
\dfrac{\partial u}{\partial t}= -(-\Delta)^\frac{\alpha}{2} u+ u \diamond
\frac{\partial
^{d+1}}{\partial t\partial x_1 \cdots \partial x_d}W \\
u(0,x)=f(x)\,.%
\end{cases}
\label{e.9.1}
\end{equation}%
The difference between the above equation and Equation
\hbox{(\ref{e.1.1})} is that here we use the Wick product
$\diamond $ (see \cite{hy}, for
example). When $\al=2$ this equation is studied in \cite{hunualart} in the case $d=1$,  $ 
H_1=1/2$,  and $H_0<1/2$ and in \cite{hns}  when $H_0, H_1, \cdots, H_d>1/2$. 
 As in these two papers, we can define the following
notion of mild solution.

\begin{definition}
\label{def1} An adapted random field $u=\{u(t,x),t\geq 0,x\in \mathbb{R}%
^{d}\}\,\ $such that $E(u ^{2}(t,x))<\infty \,$ for all $(t,x)$\
is a mild solution to Equation (\ref{e.9.1}) if \ for any
$(t,x)\in \lbrack 0,\infty
)\times \mathbb{R}^{d}$, the process $\{q_{t-s}(x-y)u(s,y)\mathbf{1}%
_{[0,t]}(s),s\geq 0,y\in \mathbb{R}^{d}\mathbb{\}}$ is Skorohod
integrable, and the following equation holds
\begin{equation}
u(t,x)=q_{t}f(x)+\int_{0}^{t}\int_{\mathbb{R}^{d}}q_{t-s}(x-y)u(s,y)\delta
W_{s,y},  \label{e.9.2}
\end{equation}
where $q_t(x)$ denotes the density function of $X_t$ and
$q_tf(x)=\int_{\mathbb{R}^d} q_t(x-y) f(y)dy$.
\end{definition}

The fact that the noise in the equation (\ref{e.9.1}) is of  multiplicative form  allows
us to find recursively an explicit expression for the Wiener chaos expansion of
the solution, which is also unique in the $L^2$ space. This approach has been extensively used in the literature. For instance, we refer to the papers \cite{hu}, \cite{hunualart},\cite{tudor} among others.
As in the paper \cite{hunualart} the mild solution $u(t,x)$ to %
\hbox{(\ref{e.9.1})} admits the following Wiener chaos expansion
\begin{equation}
u(t,x)=\sum_{n=0}^{\infty }I_{n}(f_{n}(\cdot ,t,x)),
\label{e.9.3}
\end{equation}%
where $I_n$ denotes the multiple stochastic integral with respect
to $W$ and $f_{n}(\cdot ,t,x)$ is a symmetric element in
$\mathcal{H} ^{\otimes n}$, defined explicitly as
\begin{eqnarray}
&&f_{n}(s_{1},y_{1},\ldots ,s_{n},y_{n},t,x)=\frac{1}{n!}  \label{e.9.4} \\
&&\times q_{t-s_{\sigma (n)}}(x-y_{\sigma (n)})\cdots q_{s_{\sigma
(2)}-s_{\sigma (1)}}(y_{\sigma (2)}-y_{\sigma (1)})q_{s_{\sigma
(1)}}f(y_{\sigma (1)}).  \notag
\end{eqnarray}%
In the above equation $\sigma $ denotes a permutation of
$\{1,2,\ldots ,n\}$ such that $0<s_{\sigma (1)}<\cdots <s_{\sigma
(n)}<t$. Moreover,  the solution if   exists ,   will be unique
because the kernels in the Wiener chaos expansion are uniquely
determined.

The following is the main result of this section.

\begin{theorem}\label{theorem-6.2}
Suppose that $2H_0+\frac{1}{\alpha}\sum_{i=1}^d (2H_{i}-2)>1$ and that $f$ is a
bounded
measurable function. Then the process
\begin{eqnarray}
u(t,x) &=&E^{X}\Bigg[ f(X_{t}^{x})\exp \Bigg( \int_{0}^{t}\int_{\mathbb{R}%
^d}\delta (X_{t-r}^{x}-y)W(dr,dy)  \notag \\
&&\quad -\frac12 \alpha _H \int_{0}^{t}\int_{0}^{t}|r
-s|^{2H_{0}-2} \prod_{i=1}^d \left| X^i_{r }-X^i_{s}\right|
^{2H_{i}-2}dr ds\Bigg)\Bigg] \label{e.9.5}
\end{eqnarray}
is  the unique mild solution to Equation (\ref{e.1.1}).
\end{theorem}

\begin{proof}
From Theorem \ref{exp}, we obtain that the expectation $E^X$ in Equation (\ref%
{e.9.5}) is well defined. Then, it suffices to show that the
random variable $u(t,x)$ has the Wiener chaos expansion
(\ref{e.9.3}). This can be easily proved by expanding the
exponential and then taken the expectation with respect to $X$.

Theorem \ref{teo1} implies that almost surely $\delta
(X_{t-\cdot}^{x}-\cdot)$ is and element of $\mathcal{H}$ with a
norm given by \hbox{(\ref{e.3.4})}. As a consequence, almost
surely with respect to the stable L\'evy motion $X$, we
have the following chaos expansion for the exponential factor in Equation (%
\ref{e.9.5})
\begin{eqnarray}
&& \exp \Bigg( \int_{0}^{t}\int_{\mathbb{R}^d}\delta
(X_{t-r}^{x}-y)W(dr,dy)
\notag \\
&&-\frac12 \alpha _H \int_{0}^{t}\int_{0}^{t}|r -s|^{2H_{0}-2}
\prod_{i=1}^d \left| X^i_{r }-X^i_{s}\right| ^{2H_{i}-2}dr
ds\Bigg)=\sum_{n=0}^\infty I_n(g_n)\,,  \notag
\end{eqnarray}
where $g_n$ is the symmetric element in $\mathcal{H}^{\otimes n}$
given by
\begin{eqnarray}
g_n (s_{1},y_{1},\ldots ,s_{n},y_{n},t,x)=\frac{1}{n!}\ {\delta}%
(X_{t-s_1}^x-y_1) \cdots {\delta}(X_{t-s_n}^x-y_n)\,.
\end{eqnarray}
Thus the right hand side of \hbox{(\ref{e.9.5})} admits the
following chaos expansion
\begin{eqnarray}
u(t,x)=\sum_{n=0}^\infty \frac1{n!} I_n(h_n(\cdot,t,x))\,,
\end{eqnarray}
with
\begin{eqnarray}
h_n(t,x)=E^X \left[f(X_t^x) {\delta}(X_{t-s_1}^x-y_1) \cdots {\delta}%
(X_{t-s_n}^x-y_n)\right]\,.
\end{eqnarray}
This can be regarded as a Feynman-Kac formula for the coefficients
of chaos expansion of the solution of \hbox{(\ref{e.9.1})}. To
compute the above expectation we shall use the following
\begin{eqnarray}
E^X \left[f(X_t^x) {\delta}(X_t^x-y) \big| \mathcal{F}_s\right] &=& \int_{%
\mathbb{R}^d} q_{t-s}(X_s^x -z) f(z) {\delta}(z-y) dz  \notag \\
&=& q_{t-s}(X_s^x -y)f(y)\,.  \label{e.9.7}
\end{eqnarray}
Assume that $0< s_{{\sigma}(1)}<\cdots<s_{{\sigma}(n)}<t$ for some
permutation ${\sigma}$ of $\{1, 2, \cdots, n\}$. Then conditioning
with respect to $\mathcal{F}_{t-s_{{\sigma}(1)}}$ and using the
Markov property of the L\'evy motion we have
\begin{eqnarray*}
h_n(t,x) & =& E^X\big\{ E^X \big[ {\delta}(X_{t-s_{{\sigma}(n)}}^x-y_{{\sigma%
}(n)}) \\
&&\quad \times \cdots
{\delta}(X_{t-s_{{\sigma}(1)}}^x-y_{{\sigma}(1)})
f(X_t^x) \big|\mathcal{F}_{t-s_{{\sigma}(1)}} \big]\big\} \\
& =& E^X\left[ {\delta}(X^x_{t-s_{{\sigma}(n)}} -y_{{\sigma}(n)} ) \cdots {%
\delta}(X_{t-s_{{\sigma}(1)} }^x-y_{{\sigma}(1)} )
q_{s_{{\sigma}(1)}} f(X_{t-s_{{\sigma}(1)}}^x )\right]\,.
\end{eqnarray*}
Conditioning with respect to $\mathcal{F}_{t-s_{{\sigma}(2)}}$ and using %
\hbox{(\ref{e.9.7})}, we have
\begin{eqnarray*}
h_n(t,x) &=&E^B\big\{ E^X\big[ {\delta}(X_{t-s_{{\sigma}(n)}^x}-y_{{\sigma}%
(n)} ) \\
&&\times {\delta}(X_{t-s_{{\sigma}(1)} }^x-y_{{\sigma}(1)} ) q_{s_{{\sigma}%
(1)}} f(X_{t-s_{{\sigma}(1)}}^x )\big]\big|\mathcal{F}_{t-s_{{\sigma}(2)}} %
\big\} \\
&=&E^X\Bigg\{ {\delta}(X_{t-s_{{\sigma}(n)}^x}-y_{{\sigma}(n)} ) \cdots {%
\delta}(X_{t-s_{{\sigma}(2)} }^x -y_{{\sigma}(2)} ) \\
&&\quad \times E^X\left[ {\delta}(X_{t-s_{{\sigma}(1)}
}^x-y_{{\sigma}(1)} )
q_{s_{{\sigma}(1)}} f(X_{t-s_{{\sigma}(1)}}^x )\big|\mathcal{F}_{t-s_{{\sigma%
}(2)}} \right] \Bigg\} \\
&=&E^X\Bigg[ {\delta}(X_{t-s_{{\sigma}(n)}^x}-y_{{\sigma}(n)} ) \cdots {%
\delta}(X_{t-s_{{\sigma}(2)} }^x-y_{{\sigma}(2)} ) \\
&&\quad \times p_{s_{{\sigma}(2)}-
s_{{\sigma}(1)}}(X_{t-s_{{\sigma}(2)} }^x-y_{{\sigma}(1)} ) q_{
s_{{\sigma}(1)}} f(y_{{\sigma}(1)} ) \Bigg]\,.
\end{eqnarray*}
Continuing this way we shall find out that
\begin{eqnarray}
h_n(t,x) =q_{t-s_{\sigma (n)}}(x-y_{\sigma (n)})\cdots
q_{s_{\sigma (2)}-s_{\sigma (1)}}(y_{\sigma (2)}-y_{\sigma
(1)})q_{s_{\sigma (1)}}f(y_{\sigma (1)})  \notag
\end{eqnarray}
which is the same as \hbox{(\ref{e.9.4})}.
\end{proof}

\begin{remark}
Theorem \ref{theorem-6.2} does not include the one-dimensional space-time white noise case, i.e. $d=1, H_0=H_1=\frac12$, for which equation (\ref{e.9.1}) has a unique mild solution while Feyman-Kac formula is not valid. Generally speaking, we need a stronger condition if we want to get a Feynman-Kac formula for the solution. In \cite{hunualart} and \cite{bt}, the authors studied the existence of mild solution for the stochastic heat equation with multiplicative noise.
\end{remark}

\begin{remark}
The method of this section can be applied to obtain a Feynman-Kac
formula
for the coefficients of the chaos expansion of the solution to Equation (\ref%
{e.1.1}):
\begin{equation*}
u(t,x)=\sum_{n=0}^{\infty }\frac{1}{n!}I_{n}(h_{n}(\cdot ,t,x))\,,
\end{equation*}%
with
\begin{eqnarray}
h_{n}(t,x) &=&E^{X}\Bigg[f(X_{t}^{x}){\delta }(X_{t-s_{1}}^{x}-y_{1})\cdots {%
\delta }(X_{t-s_{n}}^{x}-y_{n})  \notag \\
&&\quad \times \exp \left( \frac{1}{2}\alpha
_{H}\int_{0}^{t}\int_{0}^{t}|r-s|^{2H_{0}-2}\prod_{i=1}^{d}\left|
X_{r}^{i}-X_{s}^{i}\right| ^{2H_{i}-2}drds\right) \Bigg]\,.  \notag \\
&&
\end{eqnarray}
\end{remark}

\begin{remark}
 From the Feyman-Kac formula   we can derive the following formula for the moments of the solution analogous to (\ref{e.4.7}).
\begin{eqnarray*}
&&E\left( u(t,x)^{p}\right) =E\Bigg(\prod_{j=1}^{p}f(X_{t}^{j}+x)
  \\
&&\times \exp \left[  \alpha _{H} \sum_{j,k=1, j<k }^{p}\int_{0}^{t}%
\int_{0}^{t}|s-r|^{2H_{0}-2}%
\prod_{i=1}^{d}|X_{s}^{j,i}-X_{r}^{k,i}|^{2H_{i}-2}dsdr\right]
\Bigg),
\end{eqnarray*}%
where $p\ge 1$ is an integer, and $X^j, 1\le j\le d$, are
independent $d$-dimensional stable L\'evy motions.
\end{remark}

\section{Appendix}
%\begin{lemma}\label{lemma0}
%For any deterministic sub-additive function ${a(t)}, t \in \mathbb R^+$, the
%equality $$\lim_{t\to\infty}t^{-1}a(t) = \inf_{s>0}
%s^{-1}a(s)$$
%holds in the extended real line $[-\infty,\infty)$.
%\end{lemma}
The following lemma is from Lemma A.1 in \cite{hns}.
\begin{lemma}
\label{lemma1} Suppose $0<\beta<1, \epsilon >0, x>0, $ and that $X
$ is a standard normal random variable. Then there is a constant
$C$ independent of $x$ and $\epsilon $ (it may depend on
${\beta}$) such that
\begin{equation*}
E|x+\epsilon X|^{-\beta}\le C\min(\epsilon^{-\beta},x^{-\beta})\,.
\end{equation*}
\end{lemma}

%\begin{proof}
%It is straightforward to check that
%$K=\sup_{z\ge0}E|z+X|^{-\beta}<\infty.$ Thus
%\begin{equation}
%E|x+\epsilon X|^{-\beta}=\epsilon^{-\beta}E|\dfrac{x}{\epsilon}%
%+X|^{-\beta}\le K\epsilon^{-\beta}\,.  \label{e.2.1}
%\end{equation}
%On the other hand,
%\begin{eqnarray*}
%E|x+\epsilon X|^{-\beta} &=&\dfrac{1}{\sqrt{2\pi}} \int_\mathbb{R}%
%|x+\epsilon y|^{-\beta}e^{-\frac{y^2}{2}}dy \\
%&=&\dfrac{1}{\sqrt{2\pi}}\Bigg(\int_{\{|x+\epsilon y|>
%\frac{x}{2}\}}
%|x+\epsilon y|^{-\beta} e^{-\frac{y^2}{2}}dy \\
%&& +\int_{\{|x+\epsilon y|\le \frac{x}{2}\}} |x+\epsilon y|^{-\beta}e^{-%
%\frac{y^2}{2}}dy\Bigg)\,.
%\end{eqnarray*}
%It is easy to see that the first integral is bounded by
%$Cx^{-{\beta}}$ for some constant $C$. The second integral,
%denoted by $B$ is bounded as follows.
%\begin{eqnarray*}
%B&=&C\dfrac{1}{\epsilon}\int_{|z|<\frac{x}{2}}|z|^{-\beta}e^{-\frac{(z-x)^2%
%}{2\epsilon^2}}dz \le C\dfrac{1}{\epsilon}\int_{|z|<\frac{x}{2}%
%}|z|^{-\beta}e^{-\frac{x^2}{8\epsilon^2}}dz \\
%&=&C\dfrac{x}{\epsilon}e^{-\frac{x^2}{8\epsilon^2}}x^{-\beta} \le
%Cx^{-\beta}\,.
%\end{eqnarray*}
%Thus we have $E|x+\epsilon X|^{-\beta}\le C |x|^{-{\beta}}$.
%Combining this with \hbox{(\ref{e.2.1})}, we obtain the lemma.
%\end{proof}

\begin{lemma}
\label{lemma1'} Suppose $0<\beta<1, \epsilon >0, a>0, $ and that
$Y $ is a standard symmetric $\alpha$-stable distributed random variable.
Then there is a constant $C$ independent of $x$ and $\epsilon $
(it may depend on ${\beta}$) such that
\begin{equation*}
E|x+\epsilon Y|^{-\beta}\le C\epsilon^{-\beta}.
\end{equation*}
\end{lemma}
\begin{proof}
\begin{align*}
&E|x+\epsilon Y|^{-\beta}\\
=&\int_{\mathbb{R}}
\mathcal{F}\{|x+\epsilon\cdot|^{-\beta}\}(\xi)e^{-|\xi|^\alpha}d\xi\\
=&\int_{\mathbb{R}}
\mathcal{F}\{|\epsilon\cdot|^{-\beta}\}(\xi)e^{ix\xi}e^{-|\xi|^\alpha}d\xi\\
=&\int_{\mathbb{R}}
\frac{1}{\epsilon}\mathcal{F}\{|\cdot|^{-\beta}\}(\frac{\xi}{\epsilon})e^{ix\xi}e^{-|\xi|^\alpha}d\xi\\
=&\int_{\mathbb{R}}
\frac{1}{\epsilon}|\frac{\xi}{\epsilon}|^{\beta-1}e^{ix\xi}e^{-|\xi|^\alpha}d\xi\\
\le&\epsilon^{-\beta}\int_{\mathbb{R}}|\xi|^{\beta-1}e^{-|\xi|^\alpha}d\xi\\
\le&C\epsilon^{-\beta}.
\end{align*}
\end{proof}

The following two lemmas are Lemma A.2 and Lemma A.3 respectively in \cite{hns}.
\begin{lemma}
\label{lemma2} Suppose $\alpha\in (0,1)$. There exists a constant
$C>0$ depending only on $\alpha$ , such that
\begin{equation*}
\sup_{\epsilon,\epsilon^{\prime}}
\int_{\mathbb{R}^2}p_\epsilon(x_1+y_1)
p_{\epsilon^{\prime}}(x_2+y_2)|y_1-y_2|^{-\alpha}dy_1dy_2\le
C|x_1-x_2|^{-\alpha}.
\end{equation*}
\end{lemma}

%\begin{proof}
%We can write
%\begin{equation*}
%\int_{\mathbb{R}^2}p_\epsilon(x_1+y_1)p_{\epsilon^{%
%\prime}}(x_2+y_2)|y_1-y_2|^{-\alpha}dy_1dy_2= {\ E}\ \left(
%|{\varepsilon}
%X_1-x_1-{\varepsilon}^{\prime}X_2+x_2|^{-{\alpha}}\right) \,.
%\end{equation*}
%Thus Lemma \ref{lemma2} follows directly from Lemma \ref{lemma1}.
%\end{proof}

\begin{lemma}
\label{lemma3} Suppose $\alpha \in (0,1)$. There exists a constant
$C>0$ depending only on $\alpha$ , such that
\begin{equation*}
\sup_{\delta,\delta^{\prime}}\int_0^t\int_0^t\varphi_\delta(t-s_1-r_1)%
\varphi_{\delta^{\prime}}(t-s_2-r_2)|r_1-r_2|^{-\alpha}dr_1dr_2
\le C|s_1-s_2|^{-\alpha}
\end{equation*}
\end{lemma}

%\begin{proof}
%Since
%\begin{equation*}
%p_\delta(x)\ge p_\delta (x) I_{[0,\sqrt\delta]}(x)=\dfrac{1}{\sqrt{2\pi\delta%
%}}e^{-\frac{x^2}{2\delta}}I_{[0,\sqrt\delta]}(x)\ge \dfrac{1}{\sqrt{2\pi e }}%
%\varphi_{\sqrt\delta}(x)\,,
%\end{equation*}
%the lemma follows from Lemma $\ref{lemma2}$.
%\end{proof}

\begin{lemma}
\label{lemma4} Suppose that $2H_{0}+\frac{1}{\alpha}\sum_{i=1}^d (2H_{i}-2)>1$. Let $%
X^{1},\dots ,X^{d}$ be independent one-dimensional symmetric $\alpha$-stable
L\'evy motion. Then we have
\begin{equation*}
E\left(
\int_{0}^{t}s^{2H_{0}-2}\prod_{i=1}^{d}|X_{s}^{i}|^{2H_{i}-2}ds\right)
^{2}<\infty .
\end{equation*}
\end{lemma}

\begin{proof}
We can write
\begin{eqnarray*}
&&E\left(
\int_{0}^{t}s^{2H_{0}-2}\prod_{i=1}^{d}|X_{s}^{i}|^{2H_{i}-2}ds\right)
^{2}=2\int_{0}^{t}\int_{0}^{s}(sr)^{2H_{0}-2} \\
&&\qquad \times
\prod_{i=1}^{d}E(|X_{s}^{i}|^{2H_{i}-2}|X_{r}^{i}|^{2H_{i}-2})drds
\end{eqnarray*}%
Let $Y$ be a standard symmetric $\alpha$-stable distributed random variable.
From Lemma \ref{lemma1'}, taking into account that $2-2H_{i}<1$,
we have when $r<s$,
\begin{eqnarray}
E(|X_{r}^{i}|^{2H_{i}-2}|X_{s}^{i}|^{2H_{i}-2})
&=&E[|X_{r}^{i}|^{2H_{i}-2}E[|(s-r)^\frac{1}{\alpha}Y+x|^{2H_{i}-2}|_{x=X_{r}^{i}}]]
\notag \\
&\leq &CE[|X_{r}^{i}|^{2H_{i}-2}(s-r)^{\frac{2H_{i}-2}{\alpha}})  \notag \\
&\leq
&Cr^{\frac{2H_{i}-1}{\alpha}}(s-r)^{\frac{2H_{i}-2}{\alpha}}\,.
\label{e.2.3}
\end{eqnarray}%
As a consequence, the conclusion of the lemma follows from the
fact that
\begin{equation*}
\ \int_{0}^{t}\int_{0}^{s}r^{2H_{0}-2+
\frac{1}{\alpha}\sum_{i=1}^{d}(2H_{i}-2)}s^{2H_{0}-2}(s-r)^{\frac{1}{\alpha}\sum_{i=1}^{d}(2H_{i}-2)}drds<\infty,
\end{equation*}%
because $2H_{0}-2+ \frac{1}{\alpha}\sum_{i=1}^{d}(2H_{i}-2)>-1$
and $\frac{1}{\alpha}\sum_{i=1}^{d}(2H_{i}-2)>-1$.
\end{proof}

\begin{lemma}
\label{Lem7.6} For any $0<\beta<1$,
\begin{equation*}
E\big(|\xi|^{-\beta}-|y+\xi|^{-\beta}\big)\le C \min (1, y^2) ,
\end{equation*}
for some constant $C>0$, where $y>0$ and $\xi$ is a standard
symmetric $\alpha$-stable random variable. \end{lemma}

\begin{proof}
Notice first that $E\big(|\xi|^{-\beta}-|y+\xi|^{-\beta}\big)<C$ where $C>0$ is a constant, since $%
\lim_{y\to \infty}E|y+\xi|^{-\alpha}=0$.\\
 On the other hand,
\begin{align*}
&E\big(|\xi|^{-\beta}-|y+\xi|^{-\beta}\big)\\
=&\int_{\mathbb{R}}\mathcal{F}\left\{|\cdot|^{-\beta}-|y+\cdot|^{-\beta}\right\}(\xi)
e^{-|\xi|^\alpha}d\xi\\
=&C\int_{\mathbb{R}}|\xi|^{\beta-1}e^{-|\xi|^\alpha}(1-e^{iy\xi})d\xi\\
=&C\int_{\mathbb{R}}|\xi|^{\beta-1}e^{-|\xi|^\alpha}(1-\cos(y\xi))d\xi\\
\le&C\int_{\mathbb{R}}|\xi|^{\beta-1}e^{-|\xi|^\alpha}y^2\xi^2d\xi\\
\le&C y^2.
\end{align*}

\end{proof}

\medskip
\parindent=0pt

Xia Chen\\
Department of Mathematics\\
University of Tennessee\\
Knoxville, TN 37996-1300\\
Email: xchen@math.utk.edu

\smallskip
Yaozhong Hu  \\
Department of Mathematics \\
University of Kansas \\
Lawrence, Kansas, 66045 \\
Email: hu@math.ku.edu

\smallskip
Jian Song \\Department of Mathematics  and \\
Department of Statistics \& Actuarial Science\\
University of Hong Kong\\
Pokfulam, Hong Kong\\
Email: txjsong@hku.hk
\end{document}